\documentclass[10pt]{amsart}
\usepackage{amssymb}
\usepackage{amsmath}

\usepackage{tikz}
\usepackage{url}

\usepackage{apptools}
\usepackage{url}
\usepackage{cancel}
\usepackage{apptools}
\usepackage{soul}

\usepackage{hyperref}
\hypersetup{
	pdfstartview={XYZ null null 1.00}, 
	pdfpagemode=UseNone, 
	colorlinks,
	breaklinks, 
	linkcolor=blue,
	urlcolor=blue, 
	citecolor=blue
}

\usepackage{todonotes}
 
\usepackage{comment}

\author{Konrad Deka}
\address{Faculty of Mathematics and Computer Science, Jagiellonian University in Krakow, ul. {\L}ojasiewicza 6, 30-348 Krak\'{o}w, Poland}
\email{konrad.deka@uj.edu.pl}
\author{Dominik Kwietniak}
\address{Faculty of Mathematics and Computer Science, Jagiellonian University in Krakow, ul. {\L}ojasiewicza 6, 30-348 Krak\'{o}w, Poland}
\email{dominik.kwietniak@uj.edu.pl}

\author{Bo Peng}
\address{Department of Mathematics and Statistics, McGill University, 805 Sherbrooke St W., H3A 0B9  Montreal, Canada}
\email{bo.peng3@mail.mcgill.ca}

\author{Marcin Sabok}
\address{Department of Mathematics and Statistics, McGill University, 805 Sherbrooke St W., H3A 0B9  Montreal, Canada}
\email{marcin.sabok@mcgill.ca}

\thanks{The first author's visit to Montreal was funded by the NCN  NAWA Preludium Bis  grant PPN/STA/2021/1/00083/U/00001. The second author is funded by the NCN Preludium Bis grant 2022/47/O/ST1/03299. The third and the fourth authors are partly funded by the NSERC Discovery Grant  RGPIN-2020-05445.}
\newtheorem*{theorem*}{Theorem}
\newtheorem{theorem}{Theorem}[section]
\newtheorem{lemma}[theorem]{Lemma}
\newtheorem{claim}[theorem]{Claim}
\newtheorem{proposition}[theorem]{Proposition}
\newtheorem{corollary}[theorem]{Corollary}
\newtheorem{question}[theorem]{Question}

\theoremstyle{definition}
\newtheorem{definition}[theorem]{Definition}

\theoremstyle{remark}
\newtheorem{remark}[theorem]{Remark}

\numberwithin{equation}{section}

\newcommand{\Z}{\mathbb{Z}}
\newcommand{\B}{\{a_1, b_1, c_1, d_1, a_2, b_2, c_2, d_2\}}
\newcommand{\Aut}{\mathrm{Aut}}

\newcommand{\homeo}{\mathrm{Homeo}}
\newcommand{\hilbert}{{[0,1]^{\mathbb{N}}}}

\newcommand{\orbit}{O}
\newcommand{\Lang}{\mathrm{Lang}}
\newcommand{\gwiazdka}{\#}
\newcommand{\Ag}{A}

\newcommand{\A}
{A\setminus\{\gwiazdka\}}
\newcommand{\SA}{\mathcal{S}(\Ag)}
\newcommand{\SAb}{{\SA}_\flat}

\newcommand{\xr}{X(R)}
\newcommand{\xrb}{(X(R))_\flat}
\newcommand{\fr}{F(R)}

\newcommand{\Ae}{(\A)^*_{\text{even}}}
\newcommand{\Ao}{(\A)^*_{\text{odd}}}
\newcommand{\ce}{{\code}^*_{\text{even}}}
\newcommand{\co}{{\code}^*_{\text{odd}}}

\newcommand{\code}{B}
\newcommand{\codeodd}{\code^*_{\text{odd}}}

\newcommand{\codenohash}{\code\setminus\{\gwiazdka\}}
\newcommand{\codenohashodd}{(\code\setminus\{\gwiazdka\})^*_{\text{odd}}}

\newcommand{\SB}{\mathcal{S}(\code)}

\newcommand{\Agb}{\Ag_\flat}
\newcommand{\evenb}{(\ce)_\flat}
\newcommand{\oddb}{(\co)_\flat}

\newcommand{\alphabet}{A}

\newcommand{\class}{\mathcal{S}(\alphabet)}
\newcommand{\classflat}{\mathcal{S}(\code)_\flat}

\newcommand{\FrG}{\mathbb{F}_2}
\newcommand{\sspace}{{{\code}}^{\Z}_{\flat}}
\newcommand{\Bg}{{\{a_1, b_1, c_1, d_1, a_2, b_2, c_2, d_2 \gwiazdka\}}}

\newcommand{\Bo}{\B^*_{\text{odd}}}
\begin{document}
\title[Conjugacy of systems with specification]{Bowen's Problem 32 and the conjugacy problem for systems with specification}

\date{}

\dedicatory{}
\begin{abstract}
We show that Rufus Bowen's Problem 32 on the classification of symbolic systems with the specification property does not admit a solution that would use concrete invariants. To this end, we construct a class of symbolic systems with the specification property and show that the conjugacy relation on this class is too complicated to admit such a classification. More generally, we gauge the complexity of the classification problem for symbolic systems with the specification property. Along the way, we also provide answers to two questions related to the classification of pointed systems with the specification property: to a question of Ding and Gu related to the complexity of the classification of pointed Cantor systems with the specification property and to a question of Bruin and Vejnar related to the complexity of the classification of pointed  Hilbert cube systems with the specification property.

\end{abstract}
\maketitle

\section{Introduction}

The methods of mathematical logic can be useful in establishing  impossibility results. 
For example, a descriptive set-theoretic complexity argument was used by Wojtaszczyk and Bourgain  who solved Problem 49 from the Scottish book by establishing the non-existence of certain types of Banach spaces, see \cite[Section 33.26]{kechris2012classical}. In a similar vein, the theory of complexity of Borel equivalence relations can be used to demonstrate the impossibility of classifying certain mathematical objects. An equivalence relation $E$ on a standard Borel space $X$ is \textbf{smooth} if there exists a Borel assignment $f\colon X\to Y$ of elements of another standard Borel space $Y$ to elements of $X$ which provides a complete classification of the equivalence relation $E$, i.e., two elements $x,y\in X$ are $E$-related if and only if  $f(x)=f(y)$. The definition is broad enough to justify a Borel version of the Church--Turing thesis, namely that an isomorphism relation admits a concrete classification if and only if it is smooth.  For example, recently Panagiotopoulos, Sparling and Christodoulou \cite{panagiotopoulos2023incompleteness}  utilized this notion to show that there does not exist a  concrete observable that is complete and Borel definable, solving a long-standing problem in general relativity.


Smooth equivalence relations are actually only at the beginning of a larger hierachy of descriptive set-theoretic complexity. Given two equivalence relations $E$ and $F$ on standard Borel spaces $X$ and $Y$, respectively, we say that $E$ is \textbf{Borel-reducible} to $F$, written $E\leq_B F$ if there exists a Borel map  $f\colon X\to Y$ such that for $(x_1,x_2)\in X\times X$ we have $(x_1,x_2)\in E$ if and only if $(f(x_1),f(x_2)))\in F$. The Borel complexity of an isomorphism problem measures how complicated the problem is in comparison with other equivalence relations, when we compare equivalence relations using Borel reductions. 
Equivalence relations which are Borel-reducible to the equality $=$ on the real numbers, or anything that can be coded by real numbers, are exactly the smooth ones. However, the hierarchy goes much higher. The next step is formed by the \textbf{hyperfinite} equivalence relations \cite{dougherty1994structure}, that is those which are induced by Borel actions of the group $\mathbb{Z}$. Another successor of $=$ is defined in terms of the Friedman--Stanley version of the Turing \textbf{jump} \cite{friedman1989borel}, and is denoted by $=^+$. The jump can be iterated and all of the countable iteration of $^+$ on $=$ are induced by Borel actions of the group $S_\infty$ of permutations of $\mathbb{N}$. The class of equivalence relations reducible to actions of $S_\infty$ is quite large (cf. the recent result of Paolini and Shelah \cite{paolinishelah}) but not every Borel equivalence relation is in that class. In \cite{Hjorth} Hjorth developed the theory of turbulence and showed that turbulent group actions are not Borel reducible to any Borel $S_\infty$-action.  
The theory of complexity of equivalence relations also continues in the class of equivalence relations induced by actions of more general groups than $S_\infty$ (cf. \cite{clemens2012isometry,gao2003classification,sabok2016completeness,zielinski2016complexity}).

   A lot of effort has been put in measuring the complexity of problems arising in dynamical systems. 
The breakthrough for measure-preserving systems came with the results of Foreman and Weiss \cite{foreman2004anti} and of Foreman, Rudolph and Weiss \cite{foreman2011conjugacy}.
   In \cite{foreman2004anti} Foreman and Weiss showed that the conjugacy relation of ergodic transformations is turbulent and in \cite{foreman2011conjugacy} Foreman, Rudolph and Weiss showed that the conjugacy relation of ergodic transformations is not Borel, and thus the classification problem in ergodic theory is intractable.  In topological dynamics, Camerlo and Gao \cite{CamerloGao} proved that the conjugacy of Cantor systems is the most complicated among equivalence relation induced by actions of $S_\infty$ and for minimal Cantor systems it is shown in \cite{DGKKK} that the relation is not Borel. 

   In contrast, the conjugacy relation of symbolic systems is quite simple from the point of view of descriptive set theory. Since every isomorphism between  symbolic systems is given by a block code \cite{hedlund}, the conjugacy of symbolic systems is a \textbf{countable Borel equivalence relation}, i.e. a Borel equivalence relation whose equivalence classes are countable. 
   Clemens \cite{Clemens} proved that for a finite alphabet $A$ the topological conjugacy of symbolic systems of $A^\Z$ is a universal countable Borel equivalence relation. In \cite{GaoJacksonSeward} Gao, Jackson and Seward generalized it from subsystems of $A^\Z$ to subsystems of $A^G$ where $G$ is a countable group which is not locally finite, while for a locally finite group $G$ they showed that the conjugacy of subsystems of $A^G$ is hyperfinite. It remains unknown whether the conjugacy of minimal subsystems of $A^\Z$ is a universal countable Borel equivalence relation, which is connected to a conjecture of Thomas \cite[Conjecture 1.2]{Thomas} on the isomorphism of complete groups. In fact, it is not known \cite[Question 1.3]{SabokTsankov} whether the conjugacy relation restricted to Toeplitz systems is hyperfinite or not.

   
    A special class of symbolic systems is formed by systems with \textbf{specification}, considered by Bowen \cite{Bowen71}. A dynamical system satisfies the specification property if for every $\varepsilon>0$ we can find $k\in\mathbb{N}$ such that given any collection of finite fragments of orbits, there exists a point which is $\varepsilon$-closely following these orbit segments and takes $k$ steps to switch  between consecutive orbit segments (a formal definition is given in Section \ref{sec:specification}). Nowadays, the specification property in symbolic systems for general discrete groups goes also under the name of \textbf{strong irreducibility}, as discussed for example by Glasner, Tsankov, Weiss and Zucker \cite{glasner2021bernoulli}, Frisch and Tamuz \cite{frisch2017symbolic} and Frisch, Seward and Zucker \cite{frisch2024minimal}. Around 1970's Bowen wrote an influential list of open problems, which is now maintained on the webpage \cite{rufus.list}. Problem 32 \cite{rufus}  asks to 
   \begin{center}
       \textit{classify symbolic systems with specification}.
   \end{center}
   Several results in the direction of a classification have been obtained, as reported on the website \cite{rufus}. For example Bertrand \cite{Bertrand} proved that every symbolic system with specification is synchronized and Thomsen \cite{thomsen2006ergodic} found a connection with the theory of countable state Markov chains. However, a complete classification has not been obtained. Buhanan and Kwapisz \cite{buhanan2014cocyclic} proved a result suggesting that systems with specification are quite complicated. They considered the cocyclic shift spaces, which is a countable family of symbolic systems with specification and  proved in \cite[Theorem 1.1]{buhanan2014cocyclic} that the problem of equality of cocyclic shift spaces is undecidable. However, the equality relation is a much simpler relation than the conjugacy on such systems. In this paper we prove the following, which shows that a classification with any concrete invariants is impossible.

   \begin{theorem}\label{rufus}
       The conjugacy relation of symbolic systems with the specification property is not smooth.
   \end{theorem}

   In fact, in Theorem \ref{rufus.hyperfinite} we prove that the conjugacy relation of symbolic systems with specification is not hyperfinite and essentially the same proof shows that it is not treeable; see Theorem \ref{rufus.treeable}. 

Next, we will look the conjugacy relation of pointed systems with the specification property and consider its complexity. It turns out that in order to compute the complexity for pointed Cantor systems with the specification property, we need to solve a problem posed in a paper of Ding and Gu \cite{DG}.

  In \cite{DG} Ding and Gu consider the equivalence relation $E_{\mathrm{cs}}$ defined on the space of metrics on $\mathbb{N}$, where two metrics are equivalent if the identity map on $\mathbb{N}$ extends to a homeomorphism of the completions of $\mathbb{N}$ with respect to those two metrics. The restriction of $E_{\mathrm{cs}}$ to the set of metrics whose completion is compact is denoted by $E_{\mathrm{csc}}$. This is a natural equivalence relation from the point of view of descriptive set theory, and it is interesting to ask what is its complexity. Indeed, Ding and Gu ask \cite[Question 4.11]{DG} whether for a given countable ordinal $\alpha$ and a natural number $n$, the restriction of $E_{\mathrm{csc}}$ to the metrics whose completion is homeomorphic to $\omega^{1+\alpha}\cdot n+1$ is Borel-reducible to $=^+$. Even though this question does not seem directly connected  to the topological conjugacy of systems with specification, we find a connection between the relation $E_{\mathrel{csc}}$ and Cantor systems, using a construction coming from the work of Williams \cite{williams1} and the work of Kaya \cite{KayaSubshifts} and we answer it in the positive, by showing a slightly stronger statement.  By $\mathbb{X}_{0\textrm{-dim}}$ we denote the set of metrics on $\mathbb{N}$ whose completion is zero-dimensional. The result below implies in particular a positive answer to \cite[Question 4.11]{DG}.

\begin{theorem}\label{ding.gu} The relation 
         $E_{\mathrm{csc}}$ restricted to  $\mathbb{X}_{0\textrm{-dim}}$ is Borel bi-reducible with $=^+$.
     \end{theorem}

  
    Finally, in \cite{BV} Bruin and Vejnar 
   also studied the conjugacy relation of pointed transitive systems
    \begin{table}[ht]
    \begin{tabular}{r|c|c}
    \hline
     & homeomorphisms & pointed transitive homeomorphisms \\
    \hline
    interval & Borel-complete & $\emptyset$ \\  
    circle 
        & Borel-complete
        & $=$
      \\
    Cantor set
        & Borel-complete
        & $=^+$
        \\
    Hilbert cube
        & complete orbit e.r.
        & ?
        \\
    \end{tabular}
    \caption{Source: \cite[Table 1]{BV}.}
    \end{table}
    and asked \cite[Table 1, Question 5.6]{BV} about the complexity of the conjugacy of pointed transitive homeomorphisms of the Hilbert cube. It turns out that the conjugacy of pointed transitive homeomorphisms of the Hilbert cube has the same complexity as the conjugacy relation of Hilbert cube systems with the specification property and in this paper we answer this question as follows. 

    \begin{theorem}\label{bruin.vejnar}
       The conjugacy relation of pointed transitive Hilbert cube systems is Borel bireducible with a turbulent group action.
   \end{theorem}

    In particular, the conjugacy of pointed transitive Hilbert cube systems
    is not classifiable by countable structures, as opposed to most other relations considered in \cite{BV}. 
    


\section{Notation and preliminaries}
    
    In this paper, by a \textbf{system}, we mean a pair $(X,\varphi)$ where $X$ is a compact metric space and $\varphi$ is a homeomorphism of $X$. Given a point $x \in X$, its \textbf{orbit} is $\orbit(x) = \{\varphi^n(x) \colon n \in\mathbb{Z}\}$ and its \textbf{forward orbit} is $\{\varphi^n(x) \colon n \ge 0\}$. A point $x\in X$ is called \textbf{transitive} if the forward orbit of $x$ is dense in $X$. A system $(X,\varphi)$ is \textbf{transitive} if there is a transitive point in $X$.  If $X$ has no isolated points, then a system $(X,\varphi)$ is transitive if and only if there is a point in $X$ whose orbit is dense. A \textbf{pointed transitive system} $(X,\varphi,x)$ is a transitive system $(X,\varphi)$ together with a point $x$ whose forward orbit is dense. 
    A \textbf{subsystem} of a system $(X,\varphi)$ is a system $(Y,\psi)$ such that $Y$ is a nonempty closed set satisfying $\varphi(Y)=Y$ and $\psi$ is the restriction of $\varphi$ to $Y$. 
    
     Two systems $(X,\varphi)$ and $(Y,\psi)$ are  \textbf{conjugate} if there exists a homeomorphism $\rho\colon X\to Y$ such that $\rho\varphi=\psi\rho$. Two pointed systems $(X,\varphi,x)$ and $(Y,\psi,y)$ are  \textbf{conjugate} if they are conjugate by $\rho$ such that $\rho(x)=y$.

A \textbf{factor map} from $(X,\sigma)$ to $(Y,\tau)$ is a continuous surjection $\pi$ from $X$ to $Y$, such that $\pi  \sigma=\tau \pi$. 

     We say that a dynamical system $(Y,\tau)$ is \textbf{equicontinuous}  if the family of functions $\{\tau^n:n\in\mathbb{Z}\}$ is equicontinuous. Every system $(X,\varphi)$ admits the unique \textbf{maximal equicontinuous factor} that is there exists an equicontinuous system $(Y,\tau)$ and a factor map $\pi$ from $(X,\varphi)$ to $(Y,\tau)$  such that if $\theta$ is a factor from $(X,\varphi)$ to  an  equicontinuous system $(Z,\rho)$, then there is a factor map $\eta$ from $(Y,\tau)$ to $(Z,\rho)$ satisfying $\pi=\eta \theta$.

     For any compact space metric $X$,  the shift map $\sigma \colon X^\Z \to X^\Z$ is defined for $x=(x(n))_{n\in\Z} \in X^\Z$ as $\sigma(x)=(\sigma(x)(n))_{n\in\Z}$, where $\sigma(x)(n) = x(n+1)$ for all $n \in \Z$. We call the system $(X^\Z,\sigma)$ the \textbf{full shift over $X$}.

     Given $z\in X^\Z$ and $p\in\mathbb{N}$ define the \textbf{$p$-periodic part of $z$} as 
     \[{\rm Per}_{p}(z)=\{n\in \mathbb{Z}:\forall m\in\Z\text{ if } m\equiv n({\rm mod}\,p)\text{ then } z(m)=z(n)\}\] and let ${\rm Per}(z)=\bigcup_{p\in \mathbb{N}} {\rm Per}_p(z)$. Also, we write ${\rm Aper}(z)=\mathbb{Z}\setminus {\rm Per}(z)$ for the \textbf{aperiodic part of $z$}.  A sequence $z\in X^\Z$ is a \textbf{Toeplitz sequence}, if its aperiodic part ${\rm Aper}(z)$ is empty. A subsystem of the full shift over $X$ is a \textbf{Toeplitz system} if it is equal to the closure of the orbit of a Topelitz sequence $z\in X^\Z$.

   By a \textbf{symbolic system (over $A$)} or \textbf{shift space (over $A$)} we mean a subsystem of the full shift $(A^\Z,\sigma)$ where $A$ is a finite discrete space. In such a case, the set $A$ is referred to as the \textbf{alphabet}. 
By the classical result of Curtis, Hedlund, and Lyndon, any isomorphism $\varphi$ between shift spaces of $A^\Z$ for an alphabet $A$ is given by a block code \cite{hedlund}, which means that there exist $r\in\mathbb{N}$ and $f\colon A^{2r+1}\to A$ such that for every $k\in\mathbb{Z}$ and $x=(x(n))_{n\in\Z}\in A^\Z$ we have \[\varphi(x)(k)=f(x(k-r)\ldots x(k) \ldots x(k+r)).\] 

    For any alphabet $A$ we write $A^*$ for $\bigcup_{n=0}^\infty A^n$. We refer to the elements of $A^*$ as to \textbf{words}. If $w\in A^n$, then we refer to $n$ as to the \textbf{length} of $w$ and denote it by $|w|$. We use the convention that the empty word, denoted by $\epsilon$ is the unique word of length $0$.  We write $A^*_{\mathrm{even}}$ for the set of all words of even length and $A^*_{\mathrm{odd}}$ for the set of all words of odd length.
    
    For an interval $[a,b]\subseteq\mathbb{Z}$  and $z\in A^{\mathbb{Z}}$, by $z[a,b]$, we denote the sequence $(z(a),z(a+1), \dots, z(b))\in A^{b-a+1}$.  Similarly, we set $z[a,b)=(z(a),z(a+1), \dots, z(b-1))\in A^{b-a}$.  
    
    Given a symbolic system $X \subseteq A^\Z$, its \textbf{language} $\Lang(X)$ is the collection of 
 words in $A^*$ appearing in elements of $X$:
    \[
        \Lang(X) = \left\{ x[i, j] : x \in X, i,j\in\Z \textrm{ and } i \le j \right\}\cup\{\epsilon\}.
    \]

    For a symbolic system $(X,\sigma)$ over $A$ transitivity  is equivalent to 
    the statement that for all $u, v \in \Lang(X)$, there exists $w\in A^*$ such that 
    $uwv \in \Lang(X)$, see \cite[Section 3.7.2]{blanchard}.

    Let $E$ be a countable equivalence relation on the standard Borel space $X$. We say that $E $ is \textbf{hyperfinite} if it can be written as an increasing union of Borel equivalence relations with finite equivalence classes. An equivalence relation is hyperfinite if and only if it is induced by a Borel action of the group $\mathbb{Z}$ \cite[Theorem 5.1]{dougherty1994structure}. An equivalence relation $E$ is \textbf{treeable} \cite[Definition 3.1]{JKL} if there exists a Borel graph on the vertex set $X$ which is a forest and whose connected components are the equivalence classes of $E$. 

    The group $S_\infty$ is the group of all permutations of $\mathbb{N}$. An equivalence relation is \textbf{classifiable by countable structures} if it is Borel-reducible to an action of $S_\infty$, see \cite{hjorth2000classification}. The relation $=^+$ is the relation on $\mathbb{N}^{\mathbb{N}}$ defined by $(x_n)=^+(y_n)$ if $\{x_n:n\in\mathbb{N}\}=\{y_n:n\in\mathbb{N}\}$.

    Let $G$ be a Polish group acting on a Polish space $X$ in a Borel way. We denote by $E^X_G$ the induced equivalence relation.  Given $x\in X$ and open sets $U\subseteq X$ and $V\subseteq G$ with $x\in U$ and $1\in V$, the \textbf{local $U$-$V$-orbit} of $x$, denoted $\orbit(x,U,V)$, is the set of $y\in U$ for which there exist $l\in\mathbb{N}$ and $x = x_0, x_1, \ldots , x_l = y \in U$, and $g_0, \ldots , g_{l-1}\in V$ such that $x_{i+1} = g_i \cdot x_i$ for all $0\le i < l$. An action of $G$ on $X$ is \textbf{turbulent} \cite{hjorth2000classification} if every orbit is dense and meager and every local orbit is somewhere dense.    By the Hjorth turbulence theorem \cite[Corollary 3.19]{hjorth2000classification} the equivalence relation $E^G_X$ induced by a turbulent action is not classifiable by countable structures.

   Given a compact space $X$ we write $\homeo(X)$ for the group of homeomorphisms of $X$. The group $\homeo(X)$ is a Polish group with the compact-open topology. If $d$ is a metric on $X$, then the topology is induced by the uniform metric on $\homeo(X)$, also denoted by $d$ defined as $d(f,g)=\sup\{d(f(x),g(x)):x\in X\}$.  A subset $X\subseteq\hilbert$ is a \textbf{$Z$-set} (see \cite[Chapter 6.2]{vanmill}) if for every continuous function $f\colon\hilbert\to\hilbert$ and every $\varepsilon>0$ there exists a continuous function $g\colon \hilbert\to\hilbert\setminus X$ such that $d(f,g)<\varepsilon$. 

\section{Specification}\label{sec:specification}

Bowen introduced the specification property in \cite{Bowen71} to study  Axiom A diffeomorphisms. It is a strengthening (a uniform version) of transitivity. 
Informally, a dynamical system satisfies the specification property if for every $\varepsilon>0$ we can find $k$ such that given any collection of finite fragments of orbits (orbit segments), there exists a point which follows $\varepsilon$-closely these orbit segments and takes $k$ steps to switch between consecutive orbit segments. 

\begin{definition} Let $X$ be a compact metric space and $d$ be the metric on $X$.
    Given a map $\tau\colon X\to X$, an interval $[a,b)\subseteq\mathbb{N}$ with $0\leq a<b$, and $x\in X$ we write  $\tau^{[a,b)}(x)$ for the sequence  $(\tau^i(x))_{a\leq i <b}$ and call it the \textbf{orbit segment} (of $x$ over $[a,b)$).  Let $k$ be a natural number. A \textbf{$k$-spaced specification} is a sequence of $n\ge 2$ orbit segments $(\tau^{[a_i,b_i)}(x_i))_{1\leq i\leq n}$ such that $a_i-b_{i-1}\ge k$ for $2\leq i\leq n$. Let $\varepsilon>0$. A specification  $(\tau^{[a_i,b_i)}(x_i))_{1\leq i\leq n}$ is \textbf{$\varepsilon$-traced} if there exists $y\in X$ such that $d(\tau^j(x_i),\tau^j(y))\leq \varepsilon$ for $j\in [a_i,b_i)$, for every $1\leq i\leq n$. 
\end{definition}

\begin{definition}
    A system $(X,\tau)$ has the \textbf{specification property} if for every $\varepsilon>0$ there exists $k(\varepsilon)\in\mathbb{N}$ such that every $k(\varepsilon)$-spaced specification is $\varepsilon$-traced by a point from $X$.
\end{definition}

The specification property for a symbolic system $X$ can be restated in terms of the language $\Lang(X)$ in a similar fashion as transitivity.

\begin{proposition}\label{defspec}
Let $(X,\sigma)$ be a symbolic system. The following are equivalent.
\begin{enumerate}
    \item[(i)] The system $(X,\sigma)$ has the specification property.
    \item[(ii)] There exists $k\in\mathbb{N}$ such that for every $w,u\in\Lang(X)$ there exists $v\in\Lang(X)$ with $|v|=k$ such that $wvu\in\Lang(X)$.
\end{enumerate}
\end{proposition}
\begin{proof}
In the following, $d$ refers to the standard metric on $A^\Z$ as defined e.g., in \cite[Page 3]{bruin2022topological}. 

\noindent (i)$\Rightarrow$(ii) Put $\varepsilon=\frac{1}{2}$ and find $k=k(\varepsilon)$ using the specification property. Let $w,u\in\Lang(X)$. Set $|w|=l_w$ and $|u|=l_u$. There are $x_1,x_2\in X$ containing $w$ and $u$ as subwords. Without loss of generality assume $x_1[0,l_w)=w$ and $x_2[l_w+k,l_w+l_u+k)=u$. Then the specification formed by the orbit segments $\sigma^{[0,l_w)}(x_1)$ and $\sigma^{[l_w+k,l_w+l_u+k)}(x_2)$ is $k$-spaced, so it is $\varepsilon$-traced by some $y\in X$. Now by the definition of $d$ 
we have $w=y[0,l_w)$ and $u=y[l_w+k,l_w+l_u+k)$. Put $v=y[l_w,l_w+k)$ to see that $wvu\in\Lang(X)$.  

\noindent (ii)$\Rightarrow$(i) Let $k$ be provided by (ii). Fix $\varepsilon>0$ and take $n$ such that $2^{-n-1}<\varepsilon$. We claim that $k(\varepsilon)=k+2n$ witnesses the specification property. Suppose $\left(\sigma^{[a_1,b_1)}(x_1), \ldots, \sigma^{[a_m,b_m)}(x_m)\right)$ is an $k(\varepsilon)$-spaced specification. Without loss of generality increase  $b_i$'s if necessary) we assume that $k(\varepsilon)=a_{i}-b_{i-1}$ for every  $2\le i \le m$. Write $w_i=x_i[a_i-n,b_i+n)$ for $1\le i\le m$. Applying our assumption $i-1$ times we see that there exist $v_i$ for $2\le i\le m$ of length $k$ such that $u=w_1v_2w_2\ldots v_mw_m\in \Lang(X)$. Choose $y\in X$ that contains $u$. By shifting $y$ if necessary, we can assume that $y[a_1-n,b_m+n)=u$. We easily see  that the specification $\left(\sigma^{[a_1,b_1)}(x_1),\ldots, \sigma^{[a_m,b_m)}(x_m)\right)$ is $\varepsilon$-traced by $y$.
\end{proof}

It should be noted that, in the literature, some authors also consider other notions of specification, weaker than that introduced by Bowen. For example, sometimes one may require the existence of $k\in\mathbb{N}$ such that for every $w,u\in\Lang(X)$ there exists $v\in\Lang(X)$ with $|v|\leq k$ such that $wvu\in\Lang(X)$. For the whole panorama of specification-like properties of dynamical systems see the survey \cite{kwietniak2016panorama}.

\section{A class of systems with the specification property}

In this section, we construct families of symbolic systems that will allow us to provide lower bounds for the complexity of conjugacy problem for systems with the specification property.

Assume that $\Ag$ is a finite alphabet containing a distinguished symbol $\gwiazdka$ and $\A\neq\emptyset$.
Given a set $R\subseteq\Ao$, we 
define
\[
\fr=\{\gwiazdka w \gwiazdka : w\in\Ao \text{ and } w\notin R\}.
\]
Now, for $R\subseteq\Ao$ we define the shift space 
\[X(R)=\{x\in A^\Z:\text{no word from $\fr$ appears in }x\} \subseteq A^{\mathbb{Z}}.
\]
In other words, $w\in\Ao$ can appear between two consecutive occurrences of $\gwiazdka$ in $x\in \xr$ if and only if $w\in R$. Note that any $v\in\Ae$ is allowed between two consecutive $\gwiazdka$'s in $x\in \xr$. In particular, $\gwiazdka^l$ is an allowed word for every $l\ge 1$. It follows that every word over $\A$ is allowed in $\xr$ so we always have $(\A)^\Z\subseteq \xr$. Hence, $\xr\neq\emptyset$ for every $R\subseteq\Ao$.
Let 
 \[
 \SA = \{ \xr\subseteq\Ag^{\mathbb{Z}} : R\subseteq\Ao \}.
 \]
The space $\SA$ consists of nonempty closed subsets of $\Ag^\Z$, so it is naturally endowed with the Vietoris topology. The powerset of ${\Ao}$, identified with $2^{\Ao}$, also has the natural compact metric topology and 
the function $2^{\Ao}\ni R\mapsto \xr\in\SA$ is continuous, hence $\SA$ is a compact metric space. 
 \begin{lemma}\label{lem:top-on-S}
  The function $2^{\Ao}\ni R\mapsto \xr\in\SA$ is 1-1. 
 \end{lemma}
 \begin{proof} 
Suppose that $R_1 \neq R_2$. Without loss of generality, find $w \in R_1 \setminus R_2$. Then
        $$
            \dots \gwiazdka w \gwiazdka w \gwiazdka \dots \in X(R_1) \setminus X(R_2),
        $$
    thus $ X(R_1)\not= X(R_2)$.
 \end{proof}

Our starting point is the following.

\begin{proposition}\label{specyfikacja0} 
Every symbolic system in $\SA$ has the specification property.
\end{proposition}

\begin{proof}
Fix $R\subseteq\Ao$. We claim that that for every words $w,u\in \Lang(\xr)$ there exists $v$ with $|v|=2$ 
such that $wvu\in \Lang(\xr)$. 
Then $\xr$ has the specification property by  Proposition \ref{defspec}. If $w$ ends with $\gwiazdka w'$ where $w'\in \Ao$, then we set $v'= a$ where $a\in\A$. Otherwise we take $v'= \gwiazdka$. Similarly, we set $v''=a$, when $u$ begins  with $u'\gwiazdka $ where $u'\in \Ao$. Otherwise we take $v''= \gwiazdka$. Taking $v=v'v''$, we easily see that $wvu\in\Lang(\xr)$. 
\end{proof}

In the next sections, we will use the above idea to show that  the conjugacy problem for systems with specification is highly nontrivial.  However, in order to work over over the alphabet $\{0,1\}$ we will need a more subtle version of Proposition \ref{specyfikacja0}.

To work over the alphabet $\{0,1\}$  we replace the alphabet $A$ with a finite nonempty set $\code$ that consists of nonempty words over $\{0,1\}$, that is $\code\subseteq \{0,1\}^*\setminus\{\epsilon\}$. We call $\code$ a \textbf{code} and we refer to the elements of $\code$ as to \textbf{blocks}. 
We  identify words or sequences obtained by concatenating elements of $\code$ with the corresponding words over $\{0,1\}$ as follows. Given a word $w= w(0)\ldots w(l-1)$ over the code $\code$ we write $w_\flat$ for the word over $\{0,1\}$ obtained by concatenating $w(0),\ldots, w(l-1)$.  
We write $\evenb$ for the set of all words over $\{0,1\}$ that can be written as $w_\flat$, where $w$ is a concatenation of an even  number of blocks from $\code$, that is $w\in\ce$. We write $\oddb$ for the collection of all $w_\flat$, where $w$ is a concatenation of odd number of blocks over $\code$, that is $w\in\co$. Since the empty word $\epsilon$ has length zero, we have $\epsilon\in\evenb$.  Note that the length of $w\in\evenb$ need not to be even. 

Given a bi-infinite sequence $x=(x(k))_{k\in\Z}$ with entries in a code $\code$ we write $x_\flat$ for the element of $\{0,1\}^\mathbb{Z}$ obtained by concatenating the words appearing in $x$ so that the first letter of $x(0)$ appears at the position $0$ in $x_\flat$. 
The collection of all shifts of bi-infinite sequences $x_\flat$ where $x\in\code^\Z$ is a symbolic subsystem of $\{0,1\}^\Z$, which will be denoted by $\sspace$. That is,
\[
  \sspace=
  \{
    \sigma^k(x_\flat): k\in\Z\text{ and }x\in\code^\Z
  \} 
  \subseteq \{0,1\}^\Z.
\]
Given a subsystem $X\subseteq \code^\Z$ we write $$X_\flat=\{\sigma^k(x_\flat):k\in\Z\mbox{ and }x\in X\}$$ and note that $X_\flat$ is a subsystem of $\sspace\subseteq\{0,1\}^\Z$.

Now, given a code $\code$, treating $\code$ as the alphabet in the definition of $\mathcal{S}(\code)$, we define \[
\mathcal{S}(\code)_\flat=\{X_\flat: X\in\mathcal{S}(\code)\}.
\]

\begin{definition}
    We say that a finite code $\code\subseteq \{0,1\}^*$ is \textbf{recognizable}\footnote{We follows the terminology given in \cite{dp}, although in the literature this property is also known as \textit{unique decomposeability},  see \cite[p. 4718]{Pavlov} or \emph{strong code/unambiguously coded}, see \cite[{\S}6]{bpr} or \emph{uniquely representable}, see \cite{bdwy}. } if for every $x\in\code^\Z$ and every $0\le k<|x(0)|$ there is only one way 
    to decompose $\sigma^k(x_\flat)$ as a concatenation of blocks from $\code$, that is, if $x,y\in\code^\Z$ and $x_\flat=\sigma^k(y_\flat)$ for some $0\le k<|y(0)|$, then $k=0$ and $x=y$.
\end{definition}

Note that if a code $\code$ is recognizable, then every 
sufficiently long word  $w\in\Lang(\sspace)$ can be in a unique way decomposed into blocks in $\code$ meaning that
\[
w=(pu_1u_2\ldots u_k s)_\flat,
\]
where $u_1,\ldots,u_k\in\code$ and $p,s\notin \code$, but $p$ is a proper suffix of a block in $\code$ and $s$ is a proper prefix of a block in $\code$. In particular,  $|p|,|s| < \max\{|w|:w\in\code\}$. 

For every $R\subseteq \codenohashodd$ we write $$\fr_\flat=\{(\gwiazdka w\gwiazdka)_\flat:w\in\codenohashodd \text{ and } w\notin R\}.$$  If $\code$ is a recognizable code, then 
 \[   \xr_\flat=\{x\in\sspace:\mbox{ no word from }\fr_\flat\mbox{ appears in }x\}.\]

In analogy with Lemma \ref{lem:top-on-S} we have the following result. It gives us the topology on $\classflat$. 
\begin{lemma}\label{lem:top-on-Sb}
    For any recognizable code $\code$ the function $\SB\ni X\mapsto X_\flat\in\classflat$ is 1-1.
    \end{lemma}
\begin{proof}
    By Lemma \ref{lem:top-on-S} it is enough to note that if $R_1\not=R_2$, then $X(R_1)_\flat\not=X(R_2)_\flat$. Without loss of generality, find $w \in R_1 \setminus R_2$. Then by recognizability we have
        $$
            (\dots \gwiazdka w \gwiazdka w \gwiazdka \dots)_\flat \in X(R_1)_\flat \setminus X(R_2)_\flat.
        $$
\end{proof}

\begin{definition}
    Let $\code\subseteq\{0,1\}^*$ be a code and let $\varphi\in\Aut(\code^\Z)$. We say that
$\varphi$ is \textbf{block-length-preserving} if for every $x\in\code^\Z$ we have $|x(0)|=|\varphi(x)(0)|$.
\end{definition}

Note that if $\varphi\in\Aut(\code^\Z)$ is block-length-preserving, then for every $x\in\code^\Z$ and $k\in\Z$ we have $|x(k)|=|\varphi(x)(k)|$.

\begin{lemma}
Let $\code\subseteq\{0,1\}^*$ be a recognizable code and $\varphi\in\Aut(\code^\Z)$ 
be block-length-preserving. Then there exists unique $\varphi_\flat\in\Aut(\sspace)$ such that for every $x\in\code^\Z$ we have
\begin{equation}\label{eq:phi_flat}
\varphi(x)_\flat=\varphi_\flat(x_\flat).
\end{equation}
    \end{lemma}
\begin{proof} Given $y\in \sspace$ write  
$y=\sigma^k(x_\flat)$ where $x\in\code^\Z$ and $0\le k<|x(0)|$. The only automorphism  $\varphi_\flat$ satisfying \eqref{eq:phi_flat} has to be of the form
\[
\varphi_\flat(y)=\sigma^k((\varphi(x))_\flat).
\]
It is routine to check that $\varphi_\flat$ is continuous.
If $k+1<|x(0)|$, then
\[
\varphi_\flat(\sigma(y))=\sigma^{k+1}((\varphi(x))_\flat)=\sigma(\sigma^{k}((\varphi(x))_\flat))=\sigma(\varphi_\flat(y)).
\]
Otherwise $k+1=|x(0)|$, so
\[
\varphi_\flat(\sigma(y))=(\varphi(\sigma(x)))_\flat=\sigma^{|x(0)|}((\varphi(x))_\flat)=\sigma(\varphi_\flat(y)).
\]
Above, we used that $|\varphi(x)(0)_\flat|=|x(0)_\flat|$.
Injectivity of $\varphi_\flat$ follows from recognisability. Surjectivity of $\varphi_\flat$ is a consequence of surjectivity of $\varphi$.
\end{proof}

Given a recognizable code $\code$ we will consider the symbolic systems in the class $\SAb$. However, in general the systems in the class $\SAb$ may not have the specification property despite Proposition \ref{specyfikacja0}. To see that, first note that the system consisting of a single finite periodic orbit of length grater than $1$ does not have the specification property. This is easily seen using Proposition \ref{defspec}. Indeed, if the length of the orbit is $q$, then the system has realization as a symbolic subsystem $X$ of $\{0,1,\ldots,q-1\}^\Z$ consisting of the orbit of $(\ldots,0,1,2,\ldots,q-1,0,1,2,\ldots,q-1,\ldots$) and if $u,v,w\in\Lang(X)$ are such that $uvw\in\Lang(X)$, then $|v|$ depends on the last digit of $u$ and the first digit of $w$.   Next, note that a factor of a system with the specification property also has the specification property. Therefore, a system with specification cannot have a finite orbit as a factor. Now, if $\code$ is a finite recognizable code such that the lengths of all blocks in $\code$ are divisible by $q\geq 2$, then the system $\sspace\in\SAb$ factors onto the finite orbit of length $q$. Indeed, note that for every $y\in \sspace$ there is $0\le l<\max\{|w|:w\in\Ag\}$
such that $\sigma^l(y)\in(\code^\Z)_\flat$ and taking $\theta(y)=l\mod q$ defines a factor map and therefore the system $\sspace\in\SAb$ thus does not have the specification property (it only has the  specification property relative to the periodic factor, see \cite[Definition 37]{kwietniak2016panorama} and \cite[Definition 3.1 \& Theorem 3.6]{jung}.) 

Thus, in order to prove that every symbolic system in $\SAb$ has the specification property, we need to assume that $\code$ is a finite and recognizable code such that there are at least two relatively prime numbers among the lengths of the words in $\code$. In fact, to simplify our reasoning, we will work with a finite code $\code$ with a distinguished symbol $\gwiazdka\in\code$ of odd length  $r$ and we will assume that all blocks in $\codenohash$ have the same length $q$ that is relatively prime with $r$. 

\begin{definition}
    We say that a finite code $\code\subseteq \{0,1\}^*$ is \textbf{admissible} if $\code$ is recognizable and contains a distinguished element $\gwiazdka\in \code$ such that $|\gwiazdka|$ is odd,     $\codenohash\not=\emptyset$ and every element of $\codenohash$ has the same length, which is relatively prime with $|\gwiazdka|$.
\end{definition}


\begin{theorem}\label{specyfikacja} 
If $\code$ is an admissible code, then every symbolic system in $\SB_\flat$ has the specification property.
\end{theorem}
\begin{proof}
    Write $r=|\gwiazdka|$ and let $q$ be the number relatively prime with $r$ such that $|a|=q$ for every $a\in\A$. By our assumption $r$ is odd. 
   
    We need to show that for every $R\subseteq\Ao$ the system $\xrb \in \SAb$ has the specification property. By Proposition \ref{defspec} we need to find a natural number $k_R$ such that for every words $w,u\in \Lang(\xrb)$ there exists $v$ with $|v|=k_R$ such that $wvu\in \Lang(\xrb)$. We will show that there exists a number $k$ which works as $k_R$ for every $R\subseteq\Ao$.

    Put $$k=6 \max (q, r) + (2qr - 2q - r + 1).$$ We will show that this $k$ witnesses that every system in $\SAb$ has the specification property. Fix $R\in 2^{\Ao}$. First note that for every word $z\in \Lang(\xrb)$ there exist a prefix $z_p$ 
and a suffix $z_s$ satisfying $|z_p|, |z_s| \le 3 \max (q, r)$ and so that for some word $z'$ we have $z_p z z_s = \gwiazdka z' \gwiazdka \in \Lang(\xrb)$. Fix $w,u\in \Lang(\xrb)$. Using our observation, we find prefixes $w_p,u_p$ and suffixes $w_s,u_s$    such that  for some words $w',u'$ we have 
\[w_p w w_s = \gwiazdka w' \gwiazdka \in \Lang(\xrb)\quad\mbox{and}\quad u_p u u_s = \gwiazdka u' \gwiazdka \in \Lang(\xrb)\] and $|w_p|, |w_s|, |u_p|, |u_s| \le 3 \max (q, r)$.

Write $$d=k-|w_s|-|u_p|\geq 2pr-2q-r+1.$$
Since $r$ is odd, we get that $2q$ and $r$ are relatively prime, so by the Frobenius coin problem, every number greater or equal to $2pr-2q-r+1$ can be written as $2qi+rj$ for some $i,j\geq 0$. Find $i,j\geq0$ such that $$d=2qi+rj.$$

Let $t$ be any word which is a concatenation of $2j$ blocks from $\A$, that is $|t|=2qj$. Note that the word $$v=w_s \gwiazdka^i t  u_p$$ has length $|w_s|+|u_p|+2qi+rj=|w_s|+|u_p|+d=k$ and 
\[w_p wv u u_s=w_p w w_s \gwiazdka^i t u_p u u_s=\gwiazdka w'\gwiazdka \gwiazdka^i t\gwiazdka u'\gwiazdka \in \Lang(\xrb),\] which implies that $wv u\in \Lang(\xrb)$ and ends the proof.
\end{proof}

In the sequel we will need several admissible codes $\code$ such that the size of $\codenohash$ is a power of $2$. There is an abundance of such codes, which we show in the next several lemmas.

\begin{lemma}\label{codesize2}
    There exists an admissible code $\code$ such that $\codenohash$ has size two.
\end{lemma}
\begin{proof}
Take
    \begin{align*}
    a & = 010, \\
    b & = 011, \\
    \gwiazdka & = 0.
\end{align*}
Recognizability holds because $01$ can only appear in a word in $\sspace$ at the beginning of a block from $\codenohash$.
\end{proof}

\begin{lemma}\label{codesize4}
       There exists an admissible code $\code$ such that $\codenohash$ has size four.
\end{lemma}
\begin{proof}
Take 
    \begin{align*}
    a & = 1000001, \\
    b & = 1001001, \\
    c & = 1011101, \\
    d & = 1010101, \\
    \gwiazdka & = 10001.
\end{align*}
Recognizability holds because $0110$ appears in a word in $\sspace$ only at positions where two blocks from $\code$ meet.

Alternatively, one can take
\begin{align*}
    a & = 01011111, \\
    b & = 01001111, \\
    c & = 01000111, \\
    d & = 01000011, \\
    \gwiazdka & = 0,
\end{align*}
in which case recognizability holds  because $010$ can occur in a word in $\sspace$ only at the beginning of a block from $\codenohash$.
\end{proof}

\begin{lemma}\label{codesize8}
     There exists an admissible code $\code$ such that $\codenohash$ has size eight.
\end{lemma}
\begin{proof}

    \begin{align*}
        a_1 & = 10000000001, \\
        b_1 & = 10010000001, \\
        c_1 & = 10111000001, \\
        d_1 & = 10101000001, \\
        a_2 & = 10000000001, \\
        b_2 & = 10000001001, \\
        c_2 & = 10000011101, \\
        d_2 & = 10000010101, \\
        \gwiazdka & = 100000001.
    \end{align*}
and note that recognizability holds because $0110$ appears in a word in $\sspace$ only at positions where two blocks from $\code$ meet.
\end{proof}

One can produce also more examples. In fact, we have the following.

\begin{lemma}\label{prop:generating-codes}
For every $n\ge 1$ there exists an admissible code $\code$ such that $\codenohash$ has $2^n$ distinct elements. 
\end{lemma}
\begin{proof}
Fix $n\ge 1$. Set $\gwiazdka=1$ and for every $w=w(0)\ldots w(n-1)\in 2^n$ define
\[
b_w=10 w(0) 0 w(1) 0\ldots 0 w(n-1) 0 1.
\]
In other words, to obtain $b_w$ we replace $\ast$ in the following pattern
\[
10 (\ast 0)^n 1 =10\underbrace{\ast 0 \ast 0\ldots \ast 0}_{\ast 0\text{ repeated $n$ times}} 1
\]
by the binary digits of $w$. Put $\code=\{b_w:w\in2^n\}\cup\{\gwiazdka\}$. Note that all $b_w$ have the same length
$|b_w|=2\cdot n+ 3$. The code $\code$ is recognizable because $11$ appears only at places where two blocks are adjacent.
\end{proof}

\begin{remark}
 Taking sufficiently long code words in $\code$ we may assure  that for every preassigned $\varepsilon>0$ the topological entropy of every shift space in $\SB_\flat$ is smaller than $\varepsilon$. 
\end{remark}

\section{Classification of symbolic systems with the specification property}

Throughout this section we assume that $\Ag$ is a finite alphabet containing a distinguished symbol $\gwiazdka$ satisfying $\A\neq\emptyset$. We write $\Aut(\Ag^\mathbb{Z})$ for the group of all automorphisms of the  symbolic space $\Ag^\Z$. Given a word $w=w(0)\ldots w(l-1)$ over $\Ag$, we write
\[[w]=\{x\in\Ag^{\mathbb{Z}}: x[0,l-1]=w\}\]
for the clopen set determined by $w$.

Among all automorphisms of $\Ag^\Z$, we distinguish a subgroup that plays a special role in our proofs. In particular, for every $\varphi$ in this subgroup and every $\xr\in\SA$ we have that the image of $\xr$ via $\varphi$ is again in $\SA$. 

\begin{definition}\label{def:aut-gwiazdka}
We say that $\varphi\in\Aut(\Ag^\Z)$ 
is \textbf{$\gwiazdka$-preserving} if there is a bijection \[\varphi^*\colon (\A)^*\to(\A)^*\] that preserves the word length and such that for every $w\in(\A)^*$ we have
\[
\varphi([\gwiazdka w\gwiazdka])= [\gwiazdka \varphi^*(w)\gwiazdka].
\]
We say that the corresponding length-preserving bijection $\varphi^*$ \textbf{represents $\varphi$ on $(\A)^*$}. We write $\Aut(\Ag^\Z,\gwiazdka)$ for the set of all $\gwiazdka$-preserving automorphisms $\varphi\in\Aut(\Ag^\Z)$. 
\end{definition}

It is not difficult to see that if $\varphi\in\Aut(\Ag^\Z,\gwiazdka)$, then \[\varphi([\gwiazdka])=[\gwiazdka].\]
We will use the convention that if $\varphi\in \Aut(\Ag^\Z,\gwiazdka)$, then  $\varphi^*$ denotes the length-preserving bijection of $(\A)^*$ that representes $\varphi$ on $(\A)^*$. Note that since the empty word $\epsilon$ is the only word of length $0$, we have $\varphi^*(\epsilon)=\epsilon$. 

Now, we will define an action of the group $\Aut(\Ag^\Z,\gwiazdka)$ on $2^{\Ao}$. 
An automorphism $\varphi\in\Aut(\Ag^\Z,\gwiazdka)$ that is represented on 
$\A^*$ by $\varphi^*$ acts on $R\subseteq \Ao$ via
\[
\varphi^*(R)=\{\varphi^*(w):w\in R\}\subseteq\Ao.
\] 
Using Lemma \ref{lem:top-on-S} we identify $\SA$ with $2^{\Ao}$ by identifying $\xr\in \SA$ with $R\in 2^{\Ao}$, and so we also obtain the action of the group $\Aut(\Ag^\Z,\gwiazdka)$ on $\SA$.
   
   \begin{definition}\label{def:induced-action}
   For $\varphi\in\Aut(\Ag^\Z,\gwiazdka)$ its action on $2^{\Ao}$ is given by \[2^{\Ao}\ni R\mapsto\varphi\cdot R\in 2^{\Ao}\] where $$(\varphi \cdot R)(w)=R((\varphi^*)^{-1}(w)).$$  Identifying $\SA$ with $2^{\Ao}$ we get an action of $\Aut(\Ag^\Z,\gwiazdka)$ on $\SA$, which we refer to as the \textbf{induced action} of $\Aut(\Ag^\Z,\gwiazdka)$ on $\SA$. We write $\xr\mapsto \varphi\cdot \xr$ for the induced action.      
   \end{definition}

\begin{proposition}\label{prop:induced}
If $\varphi\in\Aut(\Ag^\Z,\gwiazdka)$ and $\xr \in \SA$, then 
\begin{equation*}
\varphi\cdot \xr=
\varphi(\xr).
\end{equation*}
\end{proposition}
\begin{proof} 
First note that $\varphi\cdot \xr=X(\varphi\cdot R)$, which
follows from the definition of the induced action and the identification of $\SA$ with $2^{\Ao}$. Next, note that $X(\varphi\cdot R)=
X(\varphi^*(R))$, since $\varphi\cdot R$ is the characteristic function of $\varphi^*(R)$. Finally, $X(\varphi^*(R))=\varphi(\xr)$, which follows from the assumption that $\varphi$ is $\gwiazdka$-preserving and that $\gwiazdka\varphi^*(w)\gwiazdka$ appears in an element of $X(\varphi^*(R))$ if and only if $w\in R$ if and only if $\gwiazdka w \gwiazdka$ appears in an element of $\xr$. 
\end{proof}

 \begin{definition}
       We say that $\varphi\in\Aut(\Ag^\Z,\gwiazdka)$ is \textbf{almost trivial} if $\varphi^*(w)=w$ for almost all $w\in\A^*$. 
   \end{definition}

Note that if there exists an automorphism in $\Aut(\Ag^\Z,\gwiazdka)$ which is not almost trivial, then $\A$ must have at least two elements.

\begin{proposition}\label{pmpaction}
    \begin{enumerate}
        \item The induced action of $\Aut(\Ag^\Z,\gwiazdka)$ on $\SA$ preserves the conjugacy relation.
        \item If $\Gamma$ is a countable subgroup of $\Aut(\Ag^\Z,\gwiazdka)$ which does not contain an almost trivial element, then the induced action on $\SA$ preserves a probability measure and is a.e. free.
    \end{enumerate}
\end{proposition}
\begin{proof}
    (1) It follows directly from 
    Proposition \ref{prop:induced}. 
    
    (2) Write $\mu$ for the $(\frac{1}{2}$-$\frac{1}{2})$-Bernoulli measure on $2^{\Ao}$. This measure is pushed forwar to $\SA$ via the map $2^{\Ao}\ni\ R\mapsto \xr\in\SA$. It is routine to check that the action of $\Gamma$ on $2^{\Ao}$ preserves $\mu$, so the induced action of $\Gamma$ on $\SA$ preserves the pushforward measure. We need to show that the action is a.e. free. It is enough to do it on the level of $2^{\Ao}$ that is to show that  $\{R\in 2^{\Ao}:\gamma \cdot R=R$ for some $\gamma\neq\text{id}\}$ is $\mu$-null.
    
    Fix $\gamma\in\Gamma\setminus\{\text{id}\}$. Assume $\gamma$ is represented on  $(\A)^*$ via $\gamma^*$. 
    Since $\gamma$ is not almost trivial there  is a sequence of distinct words $z_n\in(\A)^*$ such that
    $\gamma^* (z_n) \neq z_n.$
    Note that
    $$
    \left\{ R \in 2^{\Ao} : \gamma\cdot  R = R \right\} \subseteq
    \bigcap_{n \in\mathbb{N}} \left\{ R \in 2^{\Ao} :  R(z_n) = R((\gamma^*)^{-1} (z_n)) \right\}.
    $$
    Now, the events $\big\{\{R(z_n) = R((\gamma^*)^{-1} (z_n))\}: n\in\mathbb{N}\big\}$ are independent and for every $n$ we have $\mu \left( \left\{ R \in 2^{\Ao} : R(z_n) = R((\gamma^*)^{-1} (z_n)) \right\} \right) = \frac{1}{2}$. 
    Therefore, for every $\gamma\in \Gamma$ we have \[\mu\left(\left\{ R \in 2^{\Ao} : \gamma \cdot R= R \right\}\right) = 0.\]
    
    Since $\Gamma$ is countable, the set
    $
    \bigcup_{\gamma \in \Gamma\setminus\{\text{id}\}} \left\{ R \in 2^{\Ao} : \gamma\cdot R= R \right\}
    $
    has measure $0$, and $\Gamma$  acts freely on the set 
    $$2^{\Ao}\setminus \bigcup_{\gamma \in \Gamma\setminus\{\text{id}\}}\left\{ R \in 2^{\Ao} : \gamma \cdot R= R \right\},
    $$
    which has measure $1$. 
\end{proof}

Note that if $\code\subseteq\{0,1\}^*$ is a finite recognizable code, then every $\varphi\in\Aut(\code^\Z)$ induces an automorphism $\varphi_\flat$ of $\sspace$. Composing the transformation $\varphi\mapsto\varphi_\flat$ with the induced action of $\Aut(\code^\Z,\gwiazdka)$ on $\code^\Z$ and noting that $\varphi_\flat(X_\flat)=\varphi(X)_\flat$ we obtain an action of $\Aut(\code^\Z,\gwiazdka)$ on $\SB_\flat$.

\begin{proposition}\label{nonamenable0}
     Suppose $\Ag$ is a finite alphabet with $\gwiazdka\in \Ag$. If $\Ag\setminus\{\gwiazdka\}$ has at least four elements, then there exists a nonamenable group $\Gamma<\Aut(\Ag^\Z,\gwiazdka)$ that does not contain an almost trivial element.
\end{proposition}
\begin{proof}
Assume $a,b,c,d$ are four distinct symbols in $\A$.  Fix one of these symbols, say $a$. For $s\in \{b,c,d\}$, write $\varphi_s$ for the automorphism given by a block code that for every 
$t \in \Ag \setminus \{a, s\}$ swaps every occurrence of $ts$ with $ta$.  Formally, $\varphi_s\colon\Ag^\Z\to \A^\Z$ is given by a block code $f_s\colon \Ag^3\to\Ag$ given for $xyz\in\Ag^3$ by
\[
    f_s(xyz) = 
    \begin{cases}
        a,    & \textrm{if } y = s \textrm{ and } x \notin \{s, a\} ,\\
        s,    & \textrm{if } y = a \textrm{ and } x \notin \{s, a\}, \\
        y, & \textrm{otherwise.}
    \end{cases}
\]
The following lemma and its proof is motivated by the proof of \cite[Theorem 2.4]{boyle1988}.

\begin{lemma}\label{boyle-4}
The group 
        $\Gamma=\langle \varphi_b, \varphi_c, \varphi_d \rangle<\Aut(\Ag^\Z,\gwiazdka)$ 
        is isomorphic to $\Z_2 * \Z_2 * \Z_2$ and does not contain an almost trivial element.
\end{lemma}
\begin{proof} For every $s\in\{b,c,d\}$ write $\varphi^*_{s}$  for a map that maps $w\in\A^*$ to $\varphi_s^*(w)$ where the latter word is obtained by exchanging every occurrence of $s$ with $a$ in $w$ provided that this occurrence is preceded by 
    $t\in\Ag\setminus\{a,s\}$ in the word $\gwiazdka w$. 
This implies that  $\Gamma<\Aut(\Ag^\Z,\gwiazdka)$, 
    Note also that 
    for every $t\in\{b,c,d\}$ we have $\varphi_t \varphi_t = \textrm{id}$. 
    It remains to show that if $k\in\mathbb{N}$ and $\psi = \varphi_{i_k} \varphi_{i_{k-1}} \ldots \varphi_{i_1}\in \Gamma$ 
    for some $t_1, \ldots, t_k \in \{b, c, d\}$ such that $t_j \neq t_{j + 1}$ for 
    $1\le j< k$, then $\psi$ is not almost trivial, and, in particular, $\psi\neq \text{id}$. For $n\in\mathbb{N}$ write $w_n=\underbrace{a\ldots a}_{\text{$n$ times}}$. We claim that if $n> k$, then 
        $$\psi^*(w_n)(k) = t_k \neq a\quad\mbox{and}\quad\psi^*(w_n)(l) = a\mbox{ for }l>k $$
        and, in particular $\psi^*(w_n)\not=w_n$. 
    The claim clearly holds for $k=1$. We proceed by induction, 
    assume the claim holds for some $1\le j<k$.
    Write $v_j = \varphi^*_{t_{j}} \ldots \varphi^*_{t_1}(w_n)$.
    We know that $v_j(j) = t_{j}\neq a$
    and $v_j(l) = a$ for $j<l\le n$. Together, these conditions imply $\varphi^*_{t_{j+1}}(v_j)(j+1) = 
    t_{j+1}$ and $\varphi^*_{j+1}(v_j)(l) = 
    a$ for $l>k$. Hence, our claim holds for $k=j+1$.
\end{proof}
Using Lemma \ref{boyle-4} we finish the proof of Proposition \ref{nonamenable0}, taking $\Gamma=\langle \varphi_b, \varphi_c, \varphi_d \rangle$. 
\end{proof}

\begin{corollary}
    If the alphabet $\alphabet$ consists of at least four symbols, then the conjugacy relation of symbolic subsystems of $\alphabet^\mathbb{Z}$ with the specification property is not hyperfinite
\end{corollary}
\begin{proof}
 Assume that the alphabet $\alphabet$ consists of $\gwiazdka$ and at least four other elements. The set $\SA$ consists of systems with specification by Proposition
   \ref{specyfikacja0} and by Proposition \ref{pmpaction} there exists a free probability measure preserving action of a nonamenable group on $\SB$, which preserves the conjugacy relation.  Since a free pmp Borel action of a countable nonamenable group induces a non-hyperfinite equivalence relation \cite[Proposition 1.7]{JKL} and a Borel subequivalence relation of a hyperfinite equivalence relation is also hyperfinite \cite[Proposition 5.2]{dougherty1994structure}, this implies non-hyperfiniteness of the conjugacy relation of symbolic subsystems with the specification property.
\end{proof}

The same also holds without the assumption on the size of the alphabet, namely for the alphabet $\{0,1\}$. It implies Theorem \ref{rufus}.

  \begin{corollary}\label{rufus.hyperfinite}
       The conjugacy relation of symbolic systems with the specification property is not hyperfinite.
   \end{corollary}

\begin{proof}

By Lemma \ref{codesize4} we can choose an admissible code $\code$ containing $\gwiazdka$ and at  least four other elements.  The set $\SB_\flat$ consists of systems with specification by Theorem
   \ref{specyfikacja}. Treating $\code$ as the alphabet in the definition of $\mathcal{S}(\code)$, by Proposition \ref{pmpaction} there exists a free probability measure preserving action of a nonamenable subgroup of $\Aut(\code^\Z,\gwiazdka)$ on $\SB$, which preserves the conjugacy relation.
   Consider the map $\mathcal{S}(\code) \ni X\mapsto X_\flat\in\classflat$ and note that for $\varphi\in\Aut(\code^\Z,\gwiazdka)$ and $X\in\mathcal{S}(\code)$ we have $\varphi_\flat(X_\flat)=\varphi(X)_\flat$. Thus, the induced action of $\Aut(\code^\Z,\gwiazdka)$ on the space $\sspace$ induces an action on $\classflat$. This action preserves the probability measure pushed forward to $\classflat$ from $\SB$ via the map $X\mapsto X_\flat$. By  \cite[Proposition 1.7]{JKL} and  \cite[Proposition 5.2]{dougherty1994structure},  this implies non-hyperfiniteness of the conjugacy relation of symbolic subsystems with the specification property.
\end{proof}

\section{Streamlined proof of non-smoothness and a strenghtening}


Non-hyperfiniteness of an equivalence relation always implies that it is not smooth. However, in this section, we will show a streamlined proof that the conjugacy of symbolic systems with specification is not smooth. 

We say that a countable discrete group $\Gamma$ \textbf{acts continuously} on a Polish space $X$ if for each $\gamma\in \Gamma$, the map $x \mapsto \gamma \cdot x$ is a homeomorphism of $X$. We will use the following definition due to Osin \cite{osin2021topological},  developed by Calderoni and Clay \cite{calderoni2023condensation}.
\begin{definition}
Let $E$ be a countable Borel equivalence relation on a Polish space $X$. We say that $x_0\in X$ is a \textbf{condensed point of $E$} if
\[
x_0\in\overline{\{y\in X: yEx_0 \text{ and }x_0\neq y\}},
\]
that is, $x_0$ is an accumulation point of the set $[x_0]_E\setminus \{x_0\}$.
\end{definition}
Osin \cite[Proposition 2.7]{osin2021topological} and Calderoni and Clay \cite[Proposition 2.2]{calderoni2023condensation} showed that if $\Gamma$ is a countable discrete group acting continuously on a Polish space $X$ and $E$ is the induced countable equivalence relation, then $E$ has a condensation point if and only if $E$ is not smooth.

\begin{theorem}\label{thm:condensed}
Suppose $\Ag$ is 
is an alphabet with a distinguished element $\gwiazdka\in\alphabet$ and at least two other symbols. The equivalence relation induced by the action of $\Aut(\Ag^\Z,\gwiazdka)$ on the space $\SA$ 
has a condensed point. 
\end{theorem}
\begin{proof}
Assume $\gwiazdka\in\alphabet$ is a distinguished element and let $a,b\in\A$ be  distinct symbols. 
Let $R_0\subseteq\Ao$ be the set $\{a,aaa,aaaaa,\ldots,a^{2n-1},\ldots\}$. Note that $R_0$ has the property that for each $n\ge 1$ it contains only one word over $\A$ of length  $2n-1$. We claim that the shift space $X_0=X_{R_0}$ is a condensed point for the equivalence relation induced  by the action of $\Aut(\Ag^\Z,\gwiazdka)$ on the space $\SA$.

For each $n\ge 1$ we will find a symbolic system $X_n\in\SA$ conjugate to $X_0$ via $\varphi_n\in\Aut(\Ag^\Z,\gwiazdka)$ such that 
\begin{itemize}    
    \item[(i)] the languages of $X_n$ and $X_0$ contain the same words of length up to $2n-1$
    \item[(ii)] $\gwiazdka a^{2n-1}\gwiazdka\notin \mathrm{Lang}(X_n)$.
\end{itemize}
Condition (ii) implies that $X_n\not= X_0$. Condition (i) implies that $X_n\to X_0$ as $n\to\infty$, so having constructed $X_n$ for each $n\ge 1$ will imply that $X_0$ is a condensed point of $E$.

The map $\varphi_n$ will turn every occurrence of $\gwiazdka a^{2n-1}\gwiazdka$ into  $\gwiazdka a^{n-1}ba^{n-1}\gwiazdka$ and vice versa, every  $\gwiazdka a^{n-1}ba^{n-1}\gwiazdka$ 
will become $\gwiazdka a^{2n-1}\gwiazdka$. On the other hand, $\varphi_n$ will not change the occurrences of $\gwiazdka a^{2k-1}\gwiazdka$ for $k\neq n$.  More precisely, we define the block code map $f_n\colon \Ag^{2n+1}\to\Ag$ for $v=v(0)v(1)\ldots v(2n)\in\Ag^{2n+1}$ by  
\[
    f_n(v(0)v(1)\ldots v(2n)) = 
    \begin{cases}
        a,     & \textrm{if } v=\gwiazdka a^{n-1}ba^{n-1}\gwiazdka, \\
        b,     & \textrm{if } v=\gwiazdka a^{2n-1}\gwiazdka,  \\     
        v(n), & \textrm{otherwise.}
    \end{cases}
\]


Using $f_n$ as the block code we get the map  $\varphi_n\colon\Ag^\Z\to \Ag^\Z$ given for $x\in\Ag^\Z$ by $\varphi_n(x)=y\in\Ag^\Z$, where for every $k\in\mathbb{Z}$ we have $$y(k)=f_n(x[k-n,k+n]).$$

Note that the map $\varphi_n$ is $\gwiazdka$-preserving and it is represented on $\Ag^*$ by $\varphi_n^*$ such that  $$\varphi_n^*(a^{2n-1})=a^{n-1}ba^{n-1},\quad \varphi_n^*(a^{n-1}ba^{n-1})=a^{2n-1}\quad\mbox{and}\quad \varphi_n^*(w)=w$$ for any $w\in\Ag^*\setminus\{a^{2n-1},a^{n-1}ba^{n-1}\}$. Furthermore, by 
Proposition \ref{prop:induced} we have that
$X_n=\varphi_n(X_0)=X_{R_n}$, where
\[
R_n=\left(R_0\setminus\{\gwiazdka a^{2n-1}\gwiazdka\} \right)\cup \{\gwiazdka a^{n-1}ba^{n-1}\gwiazdka\}.
\]
So $X_n$ has the desired properties.
\end{proof}

\begin{corollary}
    Suppose $\Ag$ is  
is an alphabet with at least three symbols. The conjugacy relation on the space of subsystems of $\alphabet^\mathbb{Z}$ with the specification property is not smooth.
\end{corollary}
\begin{proof}
   Assume $\gwiazdka\in\alphabet$ is a distinguished element. The set $\SA$ consists of systems with specification by Proposition
   \ref{specyfikacja0} and by Theorem \ref{thm:condensed}, the equivalence relation induced by $\Aut(\Ag^\Z,\gwiazdka)$ on the space $\SA$ has a condensed point. Note that the action of $\Aut(\Ag^\Z,\gwiazdka)$ on the space $\SA$ preserves conjugacy.
Since a subrelation of a smooth countable Borel equivalence relation is also smooth \cite[Proposition 2.1.2(i)]{realizations} this implies that the conjugacy relation on $\class$ is not smooth.
\end{proof}

The same also holds without the assumption on the size of the alphabet, namely for the alphabet $\{0,1\}$ and we can streamline the proof of Theorem \ref{rufus}.

\begin{proof}[Proof of Theorem \ref{rufus}] By Lemma \ref{codesize2} we can choose an admissible code $\code$ with a distinguished element $\gwiazdka\in\code$ and at least two distinct elements in $\codenohash$. By Theorem \ref{specyfikacja} the class $\classflat$ consists of systems with specification. Note that treating $\code$ as the alphabet in the definition of $\mathcal{S}(\code)$, we have that the map $\mathcal{S}(\code) \ni X\mapsto X_\flat\in\classflat$ is continuous and for $\varphi\in\Aut(\code^\Z,\gwiazdka)$ and $X\in\mathcal{S}(\code)$ we have $\varphi_\flat(X_\flat)=\varphi(X)_\flat$. Thus, by Theorem \ref{thm:condensed} the action of $\Aut(\sspace,\gwiazdka)$ on $\classflat$ has a condensed point. Since the action of $\Aut(\code^\Z,\gwiazdka)$ on the space $\classflat$ preserves conjugacy and a subrelation of a smooth countable Borel equivalence relation is also smooth \cite[Proposition 2.1.2(i)]{realizations},  this implies that the conjugacy relation on $\classflat$ is not smooth.
\end{proof}

On the other hand, an easy modification of the proof of Theorem \ref{rufus.hyperfinite} gives actually a stronger conclusion than non-amenability.

\begin{theorem}\label{rufus.treeable}
    The topological conjugacy of symbolic systems with the specification property is not treeable.
\end{theorem}

The proof of Theorem \ref{rufus.treeable} will follow the same lines as that of Theorem \ref{rufus.hyperfinite}. First, we state and prove the necessary ingredients. 

\begin{proposition}\label{double-nonamenable}
     Suppose $\Ag\setminus\{\gwiazdka\}$ has at least eight elements. There exists a countable group $\Gamma<Aut(\Ag^\Z,\gwiazdka)$ which contains a copy of $\FrG\times \FrG$ and does not contain an almost trivial element. 
\end{proposition}
\begin{proof}
Assume $\A$ contains eight distinct symbols denoted $a_1,b_1,c_1,d_1$ and $a_2,b_2,c_2,d_2$. For $s\in\{b_1,c_1,d_1,b_2,c_2,d_2\}$, write $f_s$ for the block code that exchanges $ts$ with 
$ty$ if either $y=a_1$, $s\in\{b_1,c_1,d_1\}$ and $t \in \Ag \setminus \{s,a_1\}$, or if $y=a_2$, $s\in\{b_2,c_2,d_2\}$ and $t \in \Ag \setminus \{s,a_2\}$. Formally, if $i\in\{1,2\}$ is such that $s\in\{b_i,c_i,d_i\}$, then the map $f_s \colon\Ag^3\to \Ag$ is given by 
$$
    f_s(xyz) = 
    \begin{cases}
        a_i,    & \textrm{if } y = s \textrm{ and } x\in \not \in \{s, a_i\}, \\
        s,    & \textrm{if } y = a_i \textrm{ and } x \not \in \{s, a_i\}, \\
        y, & \textrm{otherwise.}
    \end{cases}
$$
and we write $\varphi_s$ for the induced automorphism of $\Ag^\Z$.
Now put $\Gamma=\Gamma_8$. 
Proposition \ref{double-nonamenable} follows from Lemma \ref{boyle-8} because $\Z_2 * \Z_2 * \Z_2$ contains $\FrG$. 
\end{proof}
The following lemma is analogous to Lemma \ref{boyle-4}.
\begin{lemma}\label{boyle-8}
        
The group $\Gamma=\langle \varphi_{b_1}, \varphi_{c_1}, \varphi_{d_1}, \varphi_{b_2}, \varphi_{c_2}, \varphi_{d_2} \rangle<\Aut(\Ag^\Z,\gwiazdka)$ is isomorphic to the product $(\Z_2 * \Z_2 * \Z_2)\times (\Z_2 * \Z_2 * \Z_2)$ and contains no almost trivial element.
\end{lemma}

\begin{proof} The proof that each $\varphi_s$  is an involution, belongs to $\Aut(\Ag,\gwiazdka)$, and that for $i\in\{1,2\}$ the group 
        $\langle \varphi_{b_i}, \varphi_{c_i}, \varphi_{d_i} \rangle$
is isomorphic to $\Z_2 * \Z_2 * \Z_2$ is the same as the proof of Lemma \ref{boyle-4}. We claim that the maps $\varphi_{s}$ and $\varphi_{t}$ commute for $s\in\{b_1,c_1,d_1\}$ and $t\in\{b_2,c_2,d_2\}$. To see this, fix $x\in \Ag^\Z$ 
 and note that for $i\in\{1,2\}$ and $s\in\{b_i,c_i,d_i\}$, the point $\varphi_s(x)$ is obtained from $x$ by interchanging $s$ and $a_i$ at those positions $x(j)$ in $x$ that $x(j-1)\notin\{s,a_i\}$. This means  $\varphi_s\varphi_t(x)=\varphi_t\varphi_s(x)$ for every $x\in \Ag^\Z$. 
 It follows that the groups  $\langle \varphi_{b_1}, \varphi_{c_1}, \varphi_{d_1} \rangle$ and $\langle \varphi_{b_2}, \varphi_{c_2}, \varphi_{d_2} \rangle$ commute. It remains to show that if $\psi\in\Gamma$ is given by $\psi = \varphi_{s_k} \varphi_{s_{k-1}} \dots \varphi_{s_1}$, where $k\in\mathbb{N}$ and  $s_1, \dots, s_k\in \{b_1, c_1, d_1,b_2,c_2,d_2\}$ is a sequence such that $s_j \neq s_{j + 1}$ for $1\le j< k$, then $\psi$ is not almost trivial, in particular, $\psi\neq\text{id}$.
 
 Let $0\le m\le k$ be the number of those elements of $\{s_1,\ldots, s_k\}$ that belong to $\{b_1,c_1,d_1\}$. 
 
 If $m>0$, then since maps $\varphi_{s}$ and $\varphi_{t}$ commute for $s\in\{b_1,c_1,d_1\}$ and $t\in\{b_2,c_2,d_2\}$, we can assume $s_i\in \{b_1,c_1,d_1\}$ for $i\leq m$. For every $n>m$ take  $w_n=\underbrace{a_{1} \dots a_{1}}_{\text{$n$ times}}$ and by induction on $m$ show that
    $$\psi^*(z_n)(m) = s_m \neq a\quad\mbox{and}\quad\psi^*(z_n)(l) = a_1\mbox{ for }l>m $$
        and, in particular $\psi^*(w_n)\not=w_n$. 
        
        If $m=0$, then and the same argument shows that for $n>k$ we have $\psi^*(w_n)\not=w_n$ with $w_n=\underbrace{a_{2} \dots a_{2}}_{\text{$n$ times}}$.
\end{proof}

\begin{corollary}
    If the alphabet $\alphabet$ consists of at least nine symbols, then the conjugacy relation of symbolic systems in $\alphabet^\mathbb{Z}$ with the specification property is not treeable
\end{corollary}
\begin{proof}
    Assume the alphabet $\alphabet$ contains a symbol $\gwiazdka$ and at least eight other symbols. The set $\SA$ consists of systems with specification by Proposition
   \ref{specyfikacja0}. Using Proposition \ref{pmpaction} and Proposition \ref{double-nonamenable} we get a free probability measure preserving (pmp) actions of $\FrG\times\FrG$ $\SA$ such that the action preserves the conjugacy relation.
     By the result of Pemantle--Peres \cite{pemantle2000nonamenable} the equivalence relation induced by the action of $\FrG\times\FrG$ is not treeable . Hence, the conjugacy of systems with specification is not treeable either, since a Borel subequivalence relation of a treeable countable Borel equivalence relation is also treeable \cite[Proposition 3.3(iii)]{JKL}. 
\end{proof}
Now we give the proof of Theorem \ref{rufus.treeable}.
\begin{proof}[Proof of Theorem \ref{rufus.treeable}]

By Lemma \ref{boyle-8} we choose an admissible code $\code$ consisting of the distinguished block $\gwiazdka$ and at least eight other symbols. The set $\SB_\flat$ consists of systems with specification by Theorem
   \ref{specyfikacja}. Treating $\code$ as the alphabet in the definition of $\mathcal{S}(\code)$, by Proposition \ref{pmpaction} and Proposition \ref{double-nonamenable} there exists a free probability measure preserving action of a copy of $\FrG\times\FrG$ in $\Aut(\code^\Z,\gwiazdka)$ on $\SB$, which preserves the conjugacy relation.
   Consider the map $\mathcal{S}(\code) \ni X\mapsto X_\flat\in\classflat$ and note that for $\varphi\in\Aut(\code^\Z,\gwiazdka)$ and $X\in\mathcal{S}(\code)$ we have $\varphi_\flat(X_\flat)=\varphi(X)_\flat$. Thus, the induced action of $\Aut(\code^\Z,\gwiazdka)$ on the space $\sspace$ induces an action on $\classflat$. This action preserves the probability measure pushed forward to $\classflat$ from $\SB$ via the map $X\mapsto X_\flat$. By the result of result of Pemantle--Peres \cite{pemantle2000nonamenable} and \cite[Proposition 3.3(iii)]{JKL}, this implies non-treeability of the conjugacy relation of symbolic subsystems with the specification property.
\end{proof}


\section{Pointed systems with the specification property}

We now turn our attention to pointed systems. 
Recall that every system with the specification property is transitive \cite[Theorem 5(4)]{kwietniak2016panorama}. Thus, below we define a pointed system with the specification property as a pointed transitive system such that the system has the specification property.
\begin{definition}
A \textbf{pointed system with the specification property} is a triple $(X,\tau,x)$ such that $(X,\tau)$ is a system with specification and $x\in X$ has dense forward orbit.    
\end{definition}
Given a topological space $X$ we refer to dynamical systems with the underlying space $X$ as to \textbf{$X$-systems}.
\begin{proposition}\label{trans-red-spec}
Let $X$ be a  compact metric space. 
The topological conjugacy relation of pointed transitive $X$-systems is Borel-reducible to the topological conjugacy of pointed $X^\Z$-systems with the  specification property.
\end{proposition}

\begin{proof} We enumerate all finite sequences of natural numbers with odd  length as $a_1,a_2,\ldots$. 
For each $m\in\mathbb{N}$ we fix finite intervals (sets of consecutive natural numbers) $I_m$ and $J_m$ in such way that $I_m$ and $J_m$ have odd cardinality equal to the length of the sequence $a_m$, and the the family $\{I_m,J_m:m\in\mathbb{N}\}$ consists of pairwise disjoint sets covering $\Z$. Note that $\lim_{m\to\infty}|J_m|=\infty$.  Write $k(m)\in\Z$ for the midpoint of $I_m$ and $l(m)$ for the midpoint of $J_m$.

Given a system $\varphi\in\homeo(X)$ and $x\in X$ 
we define $\tau_\varphi\in\homeo(X)$ and $x'\in X^\Z$. 
The homeomorphism $\tau_\varphi\colon X^\mathbb{Z}\to X^\mathbb{Z}$ is defined for $z=(z(n))_{n\in\Z}\in X^\Z$ and $i\in\Z$ by \[\tau_\varphi(z)(i)=\varphi(z(i+1)).\]
The point $x'=({x'}(n))_{n\in\Z} \in X^\Z$ is defined as
\begin{equation*}
{x'}(i)=
\begin{cases}
\varphi^{-k(m)+a_m(i)}(x),& \text{if }i\in I_m,\\
\varphi^{-l(m)}(x), &\text{if } i\in J_m.
\end{cases}
\end{equation*}
Note that the map 
$\homeo(X)\times X\ni(\varphi,x)\mapsto(\tau_\varphi,x')\in\homeo(X^\Z)\times X^\Z$
is continuous.

Fix a pointed transitive system $(X,\varphi,x)$. Write $d$ for the metric on $X$ and $d'$ for a metric compatible with the product topology on $X^\Z$.

First, we show that the pointed system $(X^\mathbb{Z},\tau_\varphi,x')$ is also transitive.  Fix $z\in X^\Z$ and $\varepsilon>0$. Choose $n>0$ and $\delta>0$ such that if $y\in X^\Z$ satisfies $d(y(i),z(i))<\delta$ for $|i|\leq n$, then $d'(y,z)<\varepsilon$. 
For each $i$ with $|i|\le n$, using density of the $\varphi$-orbit of $x$ in $X$ we find an integer $a(i)>0$ such that 
\[
d(\varphi^{a(i)}(x),z(i))<\delta.
\]
This results in a sequence $(a(-n),\ldots,a(n))$ of natural numbers of odd length. Hence there is $m$ such that the sequence $(a(-n),\ldots,a(n))$ is listed as $a_m$ on our list. By our choice of $\delta$ and $n$ we have
\[
d'(\tau_\varphi^{k(m)}(x'),z)<\varepsilon.
\]
This proves that the orbit of $x'$ is dense in $X^\Z$.

Now we show that the system $(X^\mathbb{Z},\tau_\varphi)$ has the specification property. Fix $\varepsilon>0$. Choose $k>0$ and $\delta>0$ such that if $y\in X^\Z$ satisfies $d(y(i),z(i))<\delta$ for $|i|\leq k$, then $d'(y,z)<\varepsilon$. We claim that that $k(\varepsilon)=2k$ witnesses the specification property. Note that for $a<b\in\Z$ and $x,z\in X^\Z$ we have that
\begin{equation}\label{bliskosc}
    \parbox{0.8\linewidth}{\centering if$\quad z[a-k,b+k]=x[a-k,b+k],\quad$ then\\ $d'(\varphi_\tau^i(z),\varphi_\tau^i(x))<\varepsilon$ for every $a\le i<b$.}
\end{equation}

Take $n\in\mathbb{N}$ and $x_1,\ldots,x_n \in X^\Z$. Fix $a_1,b_1,\ldots,a_n,b_n$ satisfying $a_i-b_{i-1}\ge 2k$ for $2\leq i\leq n$. Consider a $2k$-spaced specification  $\left(\tau_\varphi^{[a_i,b_i)}(x_i)\right)_{1\leq i\leq n}$. Assume, without loss of generality, that $a_i-b_{i-1}= 2k$ for $2\leq i\leq n$. 
Now, find $z\in X^\Z$  such that for every $1\le j\leq n$ we have $z[a_j-k,b_j+k]=x_j[a_j-k,b_j+k]$. By \eqref{bliskosc}, this point $\varepsilon$-traces the specification $\left(\tau_\varphi^{[a_i,b_i)}(x_i)\right)_{1\leq i\leq n}$.
Hence $(X^\Z,\tau_\varphi)$ has the specification property.

It is clear that the association $(\varphi,x)\mapsto (\tau_\varphi,x')$  preserves topological conjugacy of pointed systems. It remains to show that if $(X,\varphi,x)$ and $(X,\psi,y)$ are two pointed transitive systems such that $(X^\mathbb{Z},\tau_\varphi,x')$ and $(X^\mathbb{Z},\tau_\psi,y')$  are conjugate, then $(X,\varphi,x)$ and $(X,\psi,y)$ are also conjugate. Suppose $\rho\colon X^\Z\to Y^\Z$ conjugates the two systems.

Note that the definition of $x'$ and $\tau_{\varphi}$ implies that the sequence $\tau_\varphi^{l(m)}(x')$, where $l(m)$ is the midpoint of $J_m$ converges to the sequence $\bar x$ in $X^\Z$ whose all entries are equal $x$, that is
\begin{equation}\label{lim:xxx}
    \lim_{m\to\infty} \tau_\varphi^{l(m)}(x')=\bar x= (\ldots,x,x,x,\ldots).
\end{equation}
      By the same reasoning as the one leading to \eqref{lim:xxx} we have that
    \begin{equation}\label{lim:yyy}
    \lim_{m\to\infty} \tau_\psi^{l(m)}(y')=\bar y= (\ldots,y,y,y,\ldots).
\end{equation}
and the sequence $\rho(\tau_\varphi^{l(m)}(x'))=\tau_\psi^{l(m)}(y')$
converges to the sequence $\rho(\bar x)$ as $m\to\infty$. By \eqref{lim:xxx} and \eqref{lim:yyy} we get $\rho(\bar x)=\bar y$. 
      Note also that the set of constant sequences in $X^\mathbb{Z}$ is an invariant subsystem of $(X^\Z,\tau_{\varphi})$ that is  conjugate to $(X,\varphi)$, because if $\bar z= (z(n))_{n\in\Z}$ is a constant sequence in $X^\Z$ with $z(n)=z\in X$ for every $n\in \Z$, then
\[
\tau_{\varphi}(\bar z)=\Bar{\varphi(z)}=(\ldots,\varphi(z),\varphi(z),\varphi(z),\ldots).
\]
Hence, if pointed transitive systems $(X^\mathbb{Z},\tau_\varphi,x')$ and $(Y^\mathbb{Z},\tau_\psi,y')$ are conjugate via $\rho$, then 
    the pointed subsystems generated by restricting $\tau_\varphi$, respectively $\tau_\psi$, to the closures $\overline{\orbit(\bar x)}$, respectively $\overline{\orbit(\bar y)}$, of orbits of sequences, respectively, $\bar x =(\ldots,x,x,x,\ldots)$ and $\bar y= (\ldots,y,y,y,\ldots)$ must also be conjugate. However, pointed systems $(\overline{\orbit(\bar x)},\tau_\varphi,\bar x)$, respectively, $(\overline{\orbit(\bar y)},\tau_\psi,\bar y)$ are conjugate to pointed systems $(X,\varphi,x)$ and $(Y,\psi,y)$ via a restriction of the diagonal action $\bar{\rho}=(\ldots,\rho,\rho,\rho,\ldots)$ induced by $\rho$ on $X^\Z$. 
\end{proof}

In particular, if $X$ is the Cantor set or the Hilbert cube, then since  $X^\Z$ is homeomorphic with $X$ we get that the topological conjugacy relation of pointed transitive $X$-systems is Borel bi-reducible with the topological conjugacy of pointed $X$-systems with the specification property. One direction is stated in Proposition \ref{trans-red-spec}  and other one follows from the fact that pointed systems with the specification property are transitive.

The existence of the latter reduction follows from Proposition \ref{trans-red-spec} and the fact that if $X$ is the Cantor set or the Hilbert cube, then $X^\Z$ is homeomorphic with $X$.    
     
\section{Pointed Cantor systems with the specification property}

   
In \cite{DG} Ding and Gu define the equivalence relation on the set of metrics on $\mathbb{N}$ defined by $d_1\mathrel{E_{\mathrm{sc}}}d_2$ if there exist a homeomorphism from the completion of $(\mathbb{N},d_1)$ to the completion of $(\mathbb{N},d_2)$ that is the identity on $\mathbb{N}$. The set of all metrics on $\mathbb{N}$ is here taken with the Polish space structure induced from $\mathbb{R}^{\mathbb{N}\times\mathbb{N}}$. By \cite[Proposition 2.2]{DG} $d_1\mathrel{E_{\mathrm{sc}}}d_2$ if and only if $(\mathbb{N},d_1)$ and $(\mathbb{N},d_2)$ have the same Cauchy sequences.

Ding and Gu consider the relation $E_{\mathrm{csc}}$, which is equal to $E_{\mathrm{sc}}$ restricted to the set $\mathbb{X}_\mathrm{cpt}$ of metrics on $\mathbb{N}$ whose completion is compact. The set $\mathbb{X}_\mathrm{cpt}$ is Borel in the space of all metrics on $\mathbb{N}$, see \cite{DG}. 

We write $\mathbb{X}_{0\textrm{-dim}}$ for the set of metrics on $\mathbb{N}$ whose completion is compact and zero-dimensional. The set $\mathbb{X}_{0\textrm{-dim}}$ is Borel in $\mathbb{X}_{\mathrm{cpt}}$ \cite[Proposition 5.1]{debs2018descriptive}. For every metric $d$ in $\mathbb{X}_{\mathrm{cpt}}$ we can embed $\mathbb{N}$ into the Cantor set as $(x_n)_{n\in\mathbb{N}}$ in such a way that mapping $n\mapsto x_n$ extends to a homeomorphism of the completion of $\mathbb{N}$ with respect to $d$ and the closure of $\{x_n: n\in\mathbb{N}\}$ in the Cantor set. Then two metrics $d,d'$ with associated sequences $(x_n)_{n\in\mathbb{N}}$ and $(x_n')_{n\in\mathbb{N}}$ are $E_{\mathrm{csc}}$-related if and only if the map $x_n\mapsto x_n'$ extends to a homeomorphism of the closures of $\{x_n: n\in\mathbb{N}\}$ and $\{x_n: n\in\mathbb{N}\}$ in the Cantor set.

\begin{proposition}\label{spec.red.to.csc}
    The conjugacy relation of pointed transitive Cantor systems 
    is Borel-reducible to the relation $E_{\mathrm{csc}}$ restricted to $\mathbb{X}_{0\textrm{-dim}}$.
\end{proposition}
\begin{proof}
    Given a pointed transitive Cantor system 
    $(X,\varphi,x)$, map it to the metric on $\mathbb{N}$ induced on via the map $n\mapsto\varphi^n(x)$. This is a Borel reduction since the forward orbit of $x$ is dense in $X$.
\end{proof}

In order to gauge the complexity of the conjugation of pointed Cantor systems with the specification property, we need to gauge the complexity of $E_{\mathrm{csc}}$ restricted to $\mathbb{X}_{0\textrm{-dim}}$.

For a countable ordinal of the form $\omega^\alpha\cdot n$ Ding and Gu \cite{DG} define $\mathbb{X}_{{\omega^\alpha}\cdot n}$ as the set of metrics on $\mathbb{N}$ whose completion is homeomorphic to $\omega^{1+\alpha}\cdot n + 1$ with the order topology. In \cite[Question 4.11]{DG} the authors ask whether for every $\alpha<\omega_1$ and $n<\omega$ the relation $E_{\mathrm{csc}}$ restricted to $\mathbb{X}_{{\omega^\alpha}\cdot n}$ is reducible to $=^+$.

We will show below that a slightly stronger statement of Theorem \ref{ding.gu} is true. We will use the notion of Oxtoby systems defined by Williams \cite{williams1}. 
The standard definition of an Oxtoby sequence is slightly more general than the definition of a ternary Oxtoby sequence given below. 


\begin{definition}
    Let $X$ be a compact metric space and $(x_n)_{n\in\mathbb{N}}$ be a sequence of distinct elements of $X$. Let $y_0\notin X$ be arbitrary. Inductively on $i\in\mathbb{N}$ we define $z_i\in (X\cup\{y_0\})^{\mathbb{Z}}$. Put $z_0(j)=y_0$ for every $j\in\mathbb{Z}$. Given $z_{i}$, for every $k\in\mathbb{Z}$ write
    $$J(i,k)=\{j\in [k3^i,(k+1)3^i): z_{i}(j)=y_0\}.$$
    Put
    $$
z_{i+1}(j)=
\begin{cases}
x_{i+1}\quad\mbox{if }j\in J(i,k)\mbox{ and }k\equiv 0,2\ (\bmod\ 3),\\
z_{i}(j)\quad\mbox{ otherwise}.
\end{cases}
$$
Note that $z_n$ converges to a point in $X^{\mathbb{Z}}$ and the limit does not depend on the choice of $y_0$. The limit is denoted by $z(x_n)$. A sequence $z\in X^\mathbb{N}$ is called a ternary \textbf{Oxtoby sequence} if there exists a sequence $(x_n)_{n\in\mathbb{N}}$ of distinct elements of $X$ such that $z=z(x_n)$
   \end{definition}

Stated less formally, the ternary Oxtoby sequence is an element $z\in X^{\mathbb{Z}}$ defined inductively as follows. 
    First, for every $k\in \Z$ we put $z(3k)=z(3k+2)=x_1$. After $i$ steps of the construction, write $J(i,k)$ for the set of numbers  $j\in[k3^i,(k+1)3^i)$  at which $z(j)$ has not been defined during the first $i$ steps of the construction. In the step $i+1$ we put $z(j)=x_{i+1}$ for every $j\in J(i,k)$ with $k\equiv 0,2 \mod (3)$. More concisely, the ternary Oxtoby sequence can be defined as the sequence $z \in X^\Z$ satisfying
$z(j) = x_i$ if $j \equiv (\pm 3^{i-1} - 1)/2 \pmod{3^i}$. 

The above definition is a special case of the definition given in Williams \cite{williams1} of an Oxtoby sequence, where in place of $3^i$ one can take a fast growing sequence $p_i$, see also the discussion of symbolic Oxtoby sequences (class H4) in \cite{downar-survey}. In our case, note that for every $i\in\mathbb{N}$ and every $k\in\mathbb{Z}$ element $z_i$ assumes the value $y_0$ only on the middle point of the interval $[k3^i,(k+1)3^i)$. That is, the set $J(i,k)$ consists of a single element, which is the middle point of the interval $[k3^i,(k+1)3^i)$. To simplify notation, we write $j(i,k)$ for this unique element of $J(i,k)$. 
Also, since we will not need the more general notion of an Oxtoby sequence, we will refer to ternary Oxtoby sequences simply as to Oxtoby sequences below.

\begin{figure}[h!]
\centering
\begin{tikzpicture}
\usetikzlibrary{math}

\usetikzlibrary{decorations.pathreplacing}

\node (x0) at (-3 * 0.4, 0.0)    {$\cdot$} ;
\node (x0) at (-2 * 0.4, 0.0)    {$\cdot$} ;
\node (x0) at (-1 * 0.4, 0.0)    {$\cdot$} ;
\node (00) at ( 0 * 0.4, 0.0)    {} ;
\node (01) at ( 1 * 0.4, 0.0)    {} ;
\node (02) at ( 2 * 0.4, 0.0)    {} ;
\node (03) at ( 0.1, 0.0)    {$x_1$} ;
\node (04) at ( 0.7, 0.0)    {$x_3$} ;
\node (05) at ( 1.3, 0.0)    {$x_1$} ;
\node (06) at ( 1.9, 0.0)    {$x_1$} ;
\node (07) at ( 2.5, 0.0)    {$x_2$} ;
\node (08) at ( 3.1, 0.0)    {$x_1$} ;
\node (09) at ( 3.7, 0.0)    {$x_1$} ;
\node (10) at ( 4.3, 0.0)    {$x_2$} ;
\node (11) at ( 4.9, 0.0)    {$x_1$} ;
\node (12) at ( 5.5, 0.0)    {$x_1$} ;
\node (13) at ( 6.1, 0.0)    {$x_3$} ;
\node (14) at ( 6.7, 0.0)    {$x_1$} ;
\node (15) at ( 7.3, 0.0)    {$x_1$} ;
\node (16) at (16 * 0.4, 0.0)    {} ;
\node (17) at (17 * 0.4, 0.0)    {} ;
\node (18) at (18 * 0.4, 0.0)    {} ;
\node (19) at (19 * 0.4, 0.0)    {} ;
\node (20) at (7.8, 0.0)    {$\cdot$} ;
\node (20) at (8.2, 0.0)    {$\cdot$} ;
\node (20) at (8.6, 0.0)    {$\cdot$} ;
\node      at (3.7, -1.15) {0}       ;
\draw[-] (-5 * 0.4, 0.3) -- (25 * 0.4, 0.3) ;
\draw[-] (-5 * 0.4, -0.3) -- (25 * 0.4, -0.3) ;
\draw[-] (2.2, -0.3) -- (2.2, 0.3) ;
\draw[-] (3.4, -0.3) -- (3.4, 0.3) ;
\draw[-] (5.8, -0.3) -- (5.8, 0.3) ;
\draw[-] (4.6, -0.3) -- (4.6, 0.3) ;
\draw[-] (7, -0.3) -- (7, 0.3) ;
\draw[-] (1, -0.3) -- (1, 0.3) ;
\draw[-] (-0.2, -0.3) -- (-0.2, 0.3) ;
\draw[-] (2.8, -0.3) -- (2.8, 0.3) ;
\draw[-] (4, -0.3) -- (4, 0.3) ;
\draw[-] (5.2, -0.3) -- (5.2, 0.3) ;
\draw[-] (6.4, -0.3) -- (6.4, 0.3) ;
\draw[-] (0.4, -0.3) -- (0.4, 0.3) ;
\draw[-] (1.6, -0.3) -- (1.6, 0.3) ;
\draw[-] (7.6, -0.3) -- (7.6, 0.3) ;
\draw[->] (3.7, -0.3-0.5) -- (3.7, 0.3-0.75) ;

\end{tikzpicture}
\caption{An Oxtoby sequence.}
\end{figure}
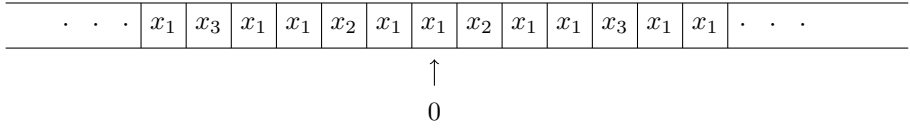

\begin{definition}
    Let $X$ be a compact metric space. An \textbf{Oxtoby system} is a subsystem of $(X^{\mathbb{Z}},\sigma)$ which is equal to $\overline{O}(z(x_n))$ for some sequence $(x_n)$ of distinct elements of $X$.
\end{definition}

The maximal equicontinuous factor of an Oxtoby system can be described quite explicitly. Suppose $z\in X^{\mathbb{Z}}$ is an Oxtoby sequence. The maximal equicontinuous factor of $\overline{O}(z)$ is $G=\varprojlim \mathbb{Z}_{3^n}$ by \cite[Theorem 2.2]{williams1}. More precisely, if we denote the element $(1,1,\dots)$ by $\Bar{1}$ and write $\widehat{1}$ for the automorphism of $G$ defined as $\widehat{1}(g)=g+\Bar{1}$, then $(G,\widehat{1})$ is the maximal equicontinuous factor of $(\overline{O}(z),S)$. Write $$A^r_i=\{\sigma^s z:s\equiv r\ (\bmod\  3^i)\}$$ where $0 \leq r< 3^i$. Williams  showed that \cite[Lemma 2.3(i)]{williams1} $\{\overline{A^r_i}: 0\le r<3^i\}$ is a partition of $\overline{O}(z)$ into clopen sets for all $i\in\mathbb{N}$ and  that \cite[Lemma 2.3(iv)]{williams1} the map $\pi \colon(\overline{O}(z),\sigma) \rightarrow (G,\widehat{1})$ defined so that \[\pi^{-1}(\{(r_i)\})= \bigcap_i \overline{A^{r_i}_i}\]  is a factor map. We refer to $\pi$ as to the \textbf{canonical maximal equicontinuous factor map}.

Below, given a subset $A\subseteq\mathbb{Z}$ and $k\in\mathbb{Z}$ we write $kA=\{ka:a\in A\}$ and $A+k=\{a+k:a\in A\}$.

\begin{claim}\label{reszta}
    If $z$ is an Oxtoby sequence, $i\in\mathbb{N}$, and $r\in[0,3^i)$, then for every $x\in A^r_i$ we have $\mathrm{Per}_{3^i}(x)=\mathrm{Per}_{3^i}(\sigma^rz)$.
\end{claim}
\begin{proof}
    Note that $\mathrm{Per}_{3^i}(z)=([0,3^i)\setminus\{j(i,0)\})+3^i\mathbb{Z}$ and for every $s\in\mathbb{Z}$ we have $\mathrm{Per}_{3^i}(\sigma^s z)=\mathrm{Per}_{3^i}(z){\color{cyan}-}s$. Since $\mathrm{Per}_{3^i}(z)+3^i\mathbb{Z}=\mathrm{Per}_{3^i}(z)$, this implies that if $s\equiv r\ (\bmod\ 3^i)$, then $\mathrm{Per}_{3^i}(\sigma^s z)=\mathrm{Per}_{3^i}(\sigma^r z)$, which ends the proof.
\end{proof}

\begin{lemma}\label{oxt}
    Let $X$ be a compact metric space, $(x_n)_{n\in\mathbb{N}}$ be a sequence of distinct elements of $X$ and $z(x_n)$ be the associated Oxtoby sequence. Write $\pi\colon \overline{O}(z(x_n))\to G$ for the canonical maximal equicontinuous factor map. Let  $u\in\overline{O}(z(x_n))$ be a non-Toeplitz sequence and let $\pi(u)=(r_n)$. If $k\in\mathrm{Aper}(u)$, then there exists $i_0\in\mathbb{N}$ such that for all $i\geq i_0$ we have $$(\sigma^{r_i}(z(x_n)))(k)=x_{i+1}.$$
\end{lemma}
\begin{proof}
     Write $z$ for $z(x_n)$. Fix $k\in\Z$. 
     \begin{claim}\label{tend}
         Both $(r_i)$ and $(3^i-r_i)$ tend to infinity.
     \end{claim}
     
     \begin{proof}
          First note that both sequences $(r_i)$ and $(3^i-r_i)$ are non-decreasing. This follows from the fact that $ r_i\in[0,3^{i})$, $r_{i+1}\equiv r_i\ (\bmod\ 3^{i})$, and $3^i|3^{i+1}$. Now, suppose $(r_i)$ does not tend to infinity. Then $(r_i)$ is eventually constant, say equal to $s$. Hence $u=\sigma^s(z)$, which is a Toeplitz sequence, a contradiction. Similarly, suppose $(3^i-r_i)$ does not converge to infinity. Then $(3^i-r_i)$ is eventually constant, say equal to $s$. Then $u=\sigma^{-s}(z)$ which is also a Toeplitz sequence, a contradiction.
     \end{proof}
     
     Thus, by Claim \ref{tend} there exists $i_0$ such that for all $i\geq i_0$ we have 
      \begin{equation}\label{nierownosc}
                -r_i<k<3^i-r_i.
\end{equation}

    We claim that this $i_0$ works, that is for all $i\geq i_0$ we have $(\sigma^{r_i}(z(x_n)))(k)=x_{i+1}$. Fix $i\geq i_0$.


    Note that  $z(j(i,0))=x_{i+1}$, and thus 
    \begin{equation}\label{rownosc}
    (\sigma^{r_i}z)(j(i,0)-r_i)=x_{i+1}.      
    \end{equation}

    By Claim \ref{reszta}  we have ${\rm Per}_{3^i}(u)={\rm Per}_{3^i}(\sigma^{r_i}(z))$  and thus 
    \begin{equation}\label{ograniczone}
        {\rm Per}_{3^i}(u)\cap [-r_i,3^i-r_i)={\rm Per}_{3^i}(\sigma^{r_i}z)\cap [-r_i,3^i-r_i).
    \end{equation}

    However, by definition, ${\rm Per}_{3^i}(z)$ and $\{j(i,0)\}$ are complementary on $[0,3^i)$,
     so 
            \begin{equation}\label{pusty}
                {\rm Per}_{3^i}(\sigma^{r_i}z)\mbox{ and }\{j(i,0)-r_i\}\mbox{ are complementary on }[-r_i,3^i-r_i).
            \end{equation} 
            Thus, by (\ref{pusty}) and (\ref{ograniczone})  we have
            \begin{equation}\label{przekroj}
                {\rm Per}_{3^i}(u)
            \mbox{ and }\{j(i,0)-r_i\}\mbox{ are complementary on } [-r_i,3^i-r_i))
            \end{equation}
             and by (\ref{przekroj}) we have 
            \begin{equation}\label{inkluzja}
                {\rm Aper}(u)\cap [-r_i,3^i-r_i)\subseteq \{j-r_i:j\in J(i,0)\}.
            \end{equation}
            By (\ref{nierownosc}), (\ref{rownosc}) and (\ref{inkluzja}) we have that $\sigma^{r_i}z(k)=x_{i+1}$ as needed. 
\end{proof}

Now we prove Theorem \ref{ding.gu}.  A similar argument is also used in the paper of the third author with Li \cite{LiPeng}
in order to compute the complexity of the conjugacy relation of pointed minimal compact systems.


\begin{proof}[Proof of Theorem \ref{ding.gu}]
    It is enough to show that $E_{\mathrm{csc}}|\mathbb{X}_{0\textrm{-dim}}$ is Borel bi-reducible with the topological conjugacy of pointed minimal Cantor systems, as the latter has the same complexity as $=^+$ by \cite{KayaSubshifts}.

    One reduction is clear. Given a pointed minimal Cantor system $(X,x,\varphi)$ we identify the forward orbit  of $x$ with $\mathbb{N}$ and endow it with the metric inherited form $X$. It is clear that two pointed minimal Cantor systems are conjugate if and only if the corresponding metrics are $E_{\mathrm{csc}}$-related.

    Now we describe a reduction of $E_{\mathrm{csc}}|\mathbb{X}_{0-\textrm{dim}}$ to the topological conjugacy of pointed minimal Cantor systems. Note that to every metric $d\in \mathbb{X}_{0-\textrm{dim}}$ we can associate in a Borel way a sequence $(x_n(d))_{n\in\mathbb{N}}$ of distinct elements of the Cantor set such that for every $d_1,d_2\in \mathbb{X}_{0-\textrm{dim}}$ the condition $d_1 \mathrel{E_{\mathrm{csc}}} d_2$ holds if and only if the map $x_n(d_1)\mapsto x_n(d_2)$ extends to a homeomorphism from $\overline{\{x_n(d_1):n\in\mathbb{N}\}}$ to $\overline{\{x_n(d_2):n\in\mathbb{N}\}}$. 
    
    Given a sequence $(x_n)$ of distinct elements of  the Cantor space we consider the Oxtoby system $\overline{O}((x_n))$. Note that the underlying space of the Oxtoby system is zero-dimensional and has no isolated points since the system is minimal. Thus, $\overline{O}(z(x_n))$ is a minimal Cantor system.

    We claim that the assignment $d\mapsto \overline{O}(z(x_n{(d)}))$ is a Borel 
    reduction of $E_{\mathrm{csc}}|\mathbb{X}_{0\textrm{-dim}}$ to the topological conjugacy of pointed Cantor minimal systems.

    ($\Rightarrow$) First, let $(x_n)$ and $(y_n)$ be two sequences of distinct elements of the Cantor space and suppose that  $\overline{O}(z(x_n))$ is conjugate to $\overline{O}(z(y_n))$ via a conjugacy  $\psi$ that sends $z(x_n)$ to $z(y_n)$. We need to show that that $(x_n)
    \mathrel{E_{\mathrm{csc}}}(y_n)$. Let $(n_i)$ be a sequence of natural numbers. By \cite[Proposition 2.2]{DG} we need to show that $x_{n_i}$ converges if and only if $y_{n_i}$ converges. Suppose  that  $x_{n_i}$ converges.  We will show that $y_{n_i}$. The other direction is analogous.

    Note that $z(x_n)$ is constructed from $(x_n)$ the same way as $z(y_n)$ is constructed from $(y_n)$. Furthermore, the systems $\overline{O}(z((x_n)))$ and $\overline{O}(z((y_n)))$ have the same equicontinuous factor $G=\varprojlim \mathbb{Z}_{3^n}$ \cite[Theorem 2.2]{williams1}. Write $\pi_1\colon \overline{O}(z((x_n)))\to G$ and $\pi_2\colon \overline{O}(z((y_n)))\to G$ for the canonical maximal equicontinuous factor maps.

    Choose a non-Toeplitz word $u\in\overline{O}(z(x_n))$ and write $\pi_1(u)=(r_n)$. We will show that $\sigma^{r_{n_i}}z(x_n)$ converges as $i\to\infty$.

    Note that if $l\in\mathrm{Per}(u)$, then $\sigma^{r_{n_i}}z(x_n)(l)$ stabilizes on $u(l)$, and hence converges. On the other hand if $k\in\mathrm{Aper}(u)$, then by Lemma \ref{oxt} we have $\sigma^{r_{n_i}}z(x_n)(k)=x_{n_i}$ for large enough $i$, and hence the sequence $\sigma^{r_{n_i}}z(x_n)(k)$ converges as $i\to\infty$. Thus  for every $m\in\mathbb{Z}$ the sequence $\sigma^{r_{n_i}}z(x_n)(m)$ converges and hence the sequence $\sigma^{r_{n_i}}z(x_n)$ converges as $i\to\infty$.
    It follows that $\sigma^{r_{n_i}}z(y_n)=\psi(\sigma^{r_{n_i}}z(x_n))$ also converges as $i\to\infty$.

Find a non-Toeplitz word $t\in \overline{O}(z((y_n)))$ such that $\pi_2(t)=(r_n)$ and let $k\in\mathrm{Aper}(t)$.  By Lemma \ref{oxt} we have $\sigma^{r_{n_i}}z(y_n)(k)=y_{n_i}$ for large enough $i$, and this sequence must converge. 
    
    ($\Leftarrow$) Let $(x_n)$ and $(y_n)$ be two sequences of elements of the Cantor space and suppose that  $(x_n)
    \mathrel{E_{\mathrm{csc}}}(y_n)$. We need to show that $\overline{O}(z(x_n))$ is conjugate to $\overline{O}(z(y_n))$ via a conjugacy sending $z(x_n)$ to $z(y_n)$. Note that for every sequence $(n_i)$ we have that $\sigma^{n_i}(z((x_n)))$ converges as $i\to\infty$ if and only if $\sigma^{n_i}(z((x_n)))$ converges as $i\to\infty$, because $z((x_n)))$ is constructed from $(x_n)$ the same way as $z(y_n)$ is constructed from $(y_n)$. Thus, we can extend the map such that $\sigma^{k}(z((x_n)))\mapsto \sigma^{k}(z((y_n)))$ for every $k\in\Z$ to a conjugacy of $\overline{O}(z((x_n)))$ and $\overline{O}(z((y_n)))$ that sends $z((x_n))$ to $z((y_n))$.
\end{proof}
\begin{corollary}
    The conjugacy relation of pointed Cantor systems with the specification property is bi-reducible with $=^+$.
\end{corollary}
\begin{proof}
    The fact that  $=^+$ is reducible to the conjugacy relation of pointed Cantor systems with the specification property follows from the result of Kaya \cite{Kaya}, the fact that every minimal system is transitive and Proposition \ref{trans-red-spec}. The other reduction follows from Theorem \ref{ding.gu} and Proposition \ref{spec.red.to.csc}.
\end{proof}
\section{Pointed transitive Hilbert cube systems}

In this section we first look at those pointed transitive Hilbert cube systems which are transitive subsystems of the full shift.

\begin{theorem}
    The action of the group $\Aut((\hilbert)^\Z)$ on the set \[\{x\in(\hilbert)^\Z: x\text{ is transitive in $((\hilbert)^{\mathbb{Z}},\sigma)$}\}\]  
    is turbulent.
\end{theorem}
\begin{proof} Fix $x\in (\hilbert)^\Z$ that is a transitive point with respect to the shift action on $(\hilbert)^\Z$. Fix a metric $d$ on $\hilbert$. Since the shift $S$ belongs to $\Aut((\hilbert)^\Z)$, the orbit of $x$ with respect to the action of $\Aut((\hilbert)^\Z)$ is also dense.  Note that the group $\homeo(\hilbert)$ can be considered to be a subgroup of $\Aut((\hilbert)^\Z)$ with $h\in \homeo(\hilbert)$ acting on $(\ldots,t_{-1},t_0,t_1,\ldots)\in (\hilbert)^\Z$ by 
\[h(\ldots,t_{-1},t_0,t_1,\ldots)=(\ldots,h(t_{-1}),h(t_0),h(t_1),\ldots).
\]
Now, the fact that every local orbit of $x$ with respect to the action of $\homeo(\hilbert)$ is somewhere dense follows from the fact that finite subsets of the Hilbert cube are $Z$-sets \cite[Lemma 6.2.3]{vanmill} and the extension theorem \cite[Theorem 6.4.6]{vanmill} saying that for every $Z$-sets $E,F\subseteq\hilbert$, every homeomorphism $f\colon E\to F$ with $d(f,\mathrm{id})<\varepsilon$ extends to a homeomorphism $g\colon\hilbert\to\hilbert$ such that $d(g,\mathrm{id})<\varepsilon$. 


Now we show that each orbit of the action of $\Aut((\hilbert)^\Z)$ on $\{x\in(\hilbert)^\Z\mid  x\mbox{ is transitive}\}$ is meager.
        Write $\Bar{0}$ for the sequence in $\hilbert$ with all coordinates equal to $0$ and $\Bar{1}$ for the sequence in $\hilbert$ with all coordinates equal to $1$. 
        Choose a sequence $(n_k)$ such that $S^{n_{k}}(x)$ converges to $x$. Write \[K=\{y\in(\hilbert)^\Z\mid  S^{n_{k}}(y) \mbox{ converges to } y\}.\]
        Clearly, $K$ depends on our choice of $x$ and $(n_{k})$ but we do not indicate it in our notation. 
        Observe that $x\in K$ and that $K$ is invariant under the action of $\Aut((\hilbert)^\Z$. We claim that $K$ is meager. To see that, consider the sets $$H_0=\{y\in(\hilbert)^\Z:\forall m\in\mathbb{N}\ \exists k\in\mathbb{N}\  d(y(n_{2k}),\Bar{0})<\frac{1}{m}\} $$ and $$H_1=\{y\in(\hilbert)^\Z:\forall m\in\mathbb{N}\ \exists k\in\mathbb{N}\  d(y(n_{2k+1}),\Bar{1})<\frac{1}{m}\}.$$ 
        Note that $K$ is disjoint from $H_0\cap H_1$, so it is enough to note that each $H_0$ and $H_1$ is comeager. Write $$H_0=\bigcap_{m\in\mathbb{N}}\bigcup_{k\in\mathbb{N}}\{y\in(\hilbert)^Z\mid d(y(n_{2k}),\Bar{0})<\frac{1}{m}\} $$ and note that for each $m$ the set $\bigcup_{k\in\mathbb{N}}\{y\in(\hilbert)^\Z\mid  d(y(n_{2k}),\Bar{0})<\frac{1}{m}\} $ is open and dense in $(\hilbert)^\Z$. The argument for $H_1$ is analogous.
\end{proof}

As noted by Bruin and Vejnar \cite{BV}, the conjugacy relation of transitive pointed Hilbert cube systems is a Borel equivalence relation, by the result of Kaya \cite{kayaambits}. In \cite[Question 5.6]{BV} Bruin and Vejnar ask about its complexity. Below we prove Theorem \ref{bruin.vejnar} by showing that this relation is Borel bi-reducible with the action of $\Aut((\hilbert)^\Z)$ on the set of shift transitive points in $(\hilbert)^\Z$. 

\begin{proof}[Proof of Theorem \ref{bruin.vejnar}]
    It is enough to note that the conjugacy of pointed transitive Hilbert cube systems is Borel bi-reducible with the action of $\Aut((\hilbert)^\Z)$ on the set $\{x\in(\hilbert)^\Z\mid x\mbox{ is transitive}\}$. 
    One direction is straightforward: given a transitive point $x$ in $(\hilbert)^\Z$ we map it to a pointed transitive Hilbert cube system $((\hilbert)^\Z),\sigma,x)$. On the other hand, given a Hilbert cube system $(\hilbert,\varphi,x)$ we map it to a transitive point in the shift in the following way. Fix a bi-infinite sequence $(n_k)_{k\in\mathbb{Z}}$ such that every finite sequence of positive integers is listed as a subsequence whose indices are consecutive integers.
    In particular, for each $j\in\mathbb{Z}$ and for every $m\in\mathbb{N}$ there is an integer $i_j(m)$ such that for every $k$ in the interval $I_j^m=[i_j(m)-m,i_j(m)+m]$ we have $n_k=j$. We now define a map $\theta$ taking a pointed transitive system $(\hilbert,\varphi,x)$ to  \[\theta((\hilbert,\varphi,x))=(\varphi^{n_k}(x))_{k\in\Z}\in(\hilbert)^\Z.\] The point $\theta((\hilbert,\varphi,x))$ is clearly transitive in $(\hilbert)^\Z$. Since the sequence $(n_k)$ is fixed, a pair of pointed transitive Hilbert cube systems $(\hilbert,\varphi,x)$ and $(\hilbert,\psi,y)$ that are conjugate through a map $\eta\colon\hilbert\to\hilbert$ with $\eta(x)=y$ is mapped to two transitive points \[\theta((\hilbert,\varphi,x))\text{ and } \theta((\hilbert,\psi,y))\] that are conjugate through $(\ldots,\eta,\eta,\eta,\ldots)\in \Aut((\hilbert)^\Z)$. It remains to show that if $(\varphi^{n_k}(x))_{k\in\Z}\in(\hilbert)^\Z$ and $(\psi^{n_k}(y))_{k\in\Z}\in(\hilbert)^\Z$ are conjugate through a map $\zeta\in \Aut((\hilbert)^\Z)$, then the pointed transitive systems $(\hilbert,\varphi,x)$ and $(\hilbert,\psi,y)$ are conjugate as well. Note that $\zeta$ must send fixed points of the shift to the fixed points of the shift, so $\zeta$ induces a homeomorphism $\eta$ of the Hilbert cube such that if $z\in \hilbert$ and $\bar z =(\ldots,z,z,z,\ldots)\in(\hilbert)^\Z$ then $\zeta(\bar z)=(\ldots,\eta(z),\eta(z),\eta(z),\ldots)$. For every $j\in\Z$ and $m\in\mathbb{N}$ we also have 
    \[
    \zeta(\sigma^{i_j(m)}(\theta((\hilbert,\varphi,x)))=\sigma^{i_j(m)}(\theta((\hilbert,\psi,y)).
    \]
    Therefore, the fixed point of the shift $(\ldots,\varphi^j(x),\varphi^j(x),\varphi^j(x),\ldots)$ that is the limit of $\sigma^{i_j(m)}(\theta((\hilbert,\varphi,x))$ as $m\to\infty$ must be mapped by $\zeta$ to the limit of $\sigma^{i_j(m)}(\theta((\hilbert,\psi,y))$, which means that for every $j\in\Z$  the homeomorphism $\eta$ maps $\varphi^j(x)$  to $\psi^j(y)$. It means that pointed transitive systems $(\hilbert,\varphi,x)$ and $(\hilbert,\psi,y)$ are conjugate.
\end{proof}
\section{Remarks and questions}

  It seems plausible that the conjugacy of both pointed transitive symbolic systems and transitive symbolic systems should be universal countable Borel equivalence relations. In \cite{kwietniak} the second named author of the present paper had announced that the classification problem of symbolic systems with the specification property is a universal countable Borel equivalence relation, but the proof contained a mistake. Therefore the following question remains open.

      \begin{question}
          Is the conjugacy relation for symbolic systems with the specification property bi-reducible with the universal countable Borel equivalence relation?
      \end{question}

Finally, one can also consider the conjugacy relation for arbitrary compact systems. Vejnar \cite{vejnar} proved that the conjugacy relation for transitive compact systems is a complete orbit equivalence relation and the third author with Li \cite{LiPeng} showed that the conjugacy relation for pointed minimal compact systems is not classifiable by countable structures. We do not know if the latter relation has the same complexity as that for pointed transitive Hilbert  cube systems.
      \begin{question}
          Does the conjugacy relation for pointed  minimal compact systems have the same complexity as for pointed transitive Hilbert  cube systems?
     \end{question}

\bibliographystyle{alpha}

\bibliography{bibliography} 

\end{document}